\documentclass[11pt,reqno]{amsart}
\usepackage{fullpage}
\usepackage{amsmath,amsthm,verbatim,amssymb,amsfonts,amscd, graphicx,bbm}
\usepackage{graphics}
\usepackage{mathtools}
\usepackage{tikz-cd}
\usepackage{mathrsfs}
\usepackage{eufrak}
\usepackage{esint}
\usepackage{color}
\usepackage{hyperref}

\newtheorem{theorem}{Theorem}
\newtheorem{proposition}[theorem]{Proposition}
\newtheorem{lemma}[theorem]{Lemma}
\newtheorem{corollary}[theorem]{Corollary}

\theoremstyle{definition}
\newtheorem{remark}{Remark}[section]

\newcommand{\cref}[1]{Corollary~\ref{c.#1}}

\numberwithin{equation}{section}
\numberwithin{theorem}{section}

\newcommand{\Z}{\mathbb{Z}}

\newcommand{\R}{\mathbb{R}}

\newcommand{\bT}{\mathbb{T}}

\newcommand{\cS}{\mathcal{S}}
\newcommand{\cD}{\mathcal{D}}
\newcommand{\cQ}{\mathcal{Q}}

\newcommand{\cK}{\mathcal{K}}
\newcommand{\cR}{\mathcal{R}}

\newcommand{\Oint}{\Omega^\eps_{\rm int}}
\newcommand{\Oext}{\Omega^\eps_{\rm ext}}

\renewcommand{\bar}{\overline}
\renewcommand{\tilde}{\widetilde}
\renewcommand{\div}{\mathrm{div}}
\newcommand{\ol}{\overline}

\newcommand{\eps}{\varepsilon}

\newcommand{\fp}{\mathfrak{p}}

\newcommand{\ran}{\mathrm{ran}}

\title{Convergence rate for the homogenization of 
diffusions in dilutely perforated domains with reflecting boundaries}

\author{Wenjia Jing}
\address{Yau Mathematical Sciences Center, Tsinghua University, Beijing 100084 and Yanqi Lake Beijing Institute of Mathematical Sciences and Applications, Beijing 101407, P.R. China}
\email{wjjing@tsinghua.edu.cn}

\date{\today}

\begin{document}

\begin{abstract}

We revisit the homogenization problem for the Poisson equation in periodically perforated domains with zero Neumann data at the boundary of the holes and prescribed Dirichlet data at the outer boundary. It is known that, if the periodicity of the holes goes to zero but their volume fraction remains fixed and positive, the limit problem is a Dirichlet boundary value problem posed in the domain without the holes, and the effective diffusion coefficients are non-trivially modified; if that volume fraction goes to zero instead, i.e.\,the holes are dilute, the effective operator remains the Laplacian (that is, unmodified). Our main results contain the study of a ``continuity" in those effective models with respect to the volume fraction of the holes and some new convergence rates for homogenization in the dilute setting. Our method explores the classical two-scale expansion ansatz and relies on asymptotic analysis of the rescaled cell problems using layer potential theory. 

\smallskip


\smallskip

\noindent{\bf Mathematics subject classification (MSC 2010)}: 35B27, 35J08

\end{abstract}

\maketitle


\section{Introduction}

In this paper, we revisit and further explore the homogenization theory for the following mixed boundary value  problem in periodically perforated domains:
\begin{equation}
\label{eq:hetnp}
\left\{
\begin{aligned}
&-\Delta u^\eps(x) = f(x), &\quad &x\in \Omega^{\eps,\eta},\\
&N_x\cdot \nabla u^\eps(x) = 0, &\quad &x\in \partial \Omega^{\eps,\eta}_{\rm int},\\
&u^\eps(x) = g(x), &\quad &x \in \partial \Omega^{\eps,\eta}_{\rm ext}.
\end{aligned}
	\right.
\end{equation}
Here, $\Omega^{\eps,\eta}$ is a highly oscillatory perforated domain which is obtained by removing a periodic array of holes from an open bounded domain $\Omega \subseteq \R^d$, $d\ge 2$. The small parameter $\eps$ denotes the periodicity of holes, $\eta = \eta_\eps$ denotes the ratio with respect to $\eps$ of the length scale of the holes; the volume fraction of the holes is hence of order $O(\eta^d)$. The boundary of $\Omega^{\eps,\eta}$ has two parts: an interior part $\partial \Omega^{\eps,\eta}_{\rm int}$ consisting boundaries of the holes inside $\Omega$, and an external part $\partial \Omega^{\eps,\eta}_{\rm ext}$; $N_x$ is the normal vector along the inner boundary at $x\in \partial \Omega^{\eps,\eta}$ pointing to the outside of the holes. See Section \ref{sec:geosetup} for a detailed geometric set-up.

The above problem models the steady state of heat conduction in a piece of homogeneous and isotropic material occupying a part of $\Omega$ that is outside an array of isolated heat barriers (the holes); $f$ models heat sources, the Dirichlet data $g$ is the prescribed heat distribution at the outer boundary, and the Neumann boundary condition says the holes are heat insulators. From a probabilistic view point, $u^\eps$ is given by the Feymann-Kac formula
\begin{equation*}
u^\eps(x) = \mathbf{E}^x \left\{g(X_{\tau^\eps}) - \frac12\int_0^{\tau^\eps} f(X_t)dt\right\}, \qquad x\in \Omega^{\eps,\eta}.
\end{equation*}
Here $(X_t)_{t\ge 0}$ is a diffusion process in $\R^d$ that evolves like the standard Brownian motion unless when it hits the boundary of a (interior) hole where it gets reflected (with respect to the outer normal of the boundary). $\mathbf{E}^x$ stands for the average over processes that start from $x \in \Omega^\eps$. Hence, $u^\eps(x)$ is the averaged value of the sum of a running cost due to $f$ and a termination cost due to $g$ in the life-span $[0,\tau^\eps)$ of $X$, where $\tau^\eps$ is the exit time when $X$ first hits $\partial \Omega^{\eps,\eta}_{\rm ext}$.   

The problem \eqref{eq:hetnp} is posed in a highly oscillatory domain, and is an example of PDEs with highly heterogeneous data. Such problems are common models for applied physics and engineering, such as in material science, atmospheric science, reservoir engineering, etc. The Laplacian in singularly perturbed domains was studied as early as by \cite{MR377303} in the context of resolvent analysis and scattering theory. Brownian motions in domains with small absorbing holes, which corresponds to zero Dirichlet data at $\partial \Oint$, was studied in \cite{MR609184} using probabilistic tools. Periodic homogenization of elliptic problems in perforated domains dates back at least to \cite{MR548785}; the case of zero Dirichlet data on the holes attracted particular attention because, as shown in \cite{CioMur-1,Kacimi_Murat,Allaire91-1,Allaire91-2}, there is a \emph{critical} scale for the size of the holes compared with the periodicity of the holes. At this scale the collective effect of the holes starts to appear in the macroscopic model; below the critical scale the holes play no role and above the critical scale the effective model is more dramatically modified. When Neumann data are prescribed at the boundary of the holes, the picture is very different. Conca and Donato showed in \cite{MR974289} that only for \emph{non-zero} (periodic) Neumann data there is a critical scale and it is much larger compared with the aforementioned Dirichlet case. For zero Neumann data like in \eqref{eq:hetnp}, the effective operator is either modified from the Laplacian if the hole size is comparable with the periodicity (see, e.g.\,\cite{AllMur}),  or the effective model is unperturbed if the holes are dilute. 

In this paper, we focus on the case with zero Neumann data at the boundary of the holes, and investigate the continuous transition of the effective models with respect to the relative smallness of the holes, and we establish quantitative convergence results for homogenization when the holes are dilute. The proofs are based on a careful study of the rescaled cell-problems using \emph{layer potential} techniques, a method started in \cite{MR4075336,MR4172687,JLP-stokes} for the case of Dirichlet data in the holes, and exhibit the effectiveness of this approach. Before we focus on the problem \eqref{eq:hetnp}, let us mention that homogenization in perforated domains and in similar geometric settings has gained a great deal of attention recently with new convergence rates and uniform regularity \cite{shen2020sharp,shen2021compactness,shen2021homogenization,MR3659366,MR4240768,KLS13_Neumann} following the framework of \cite{AL87_Lp,AL91_Lp,KLS14_GN,KLS12_ARMA}, derivations in the non-periodic settings \cite{MR4020526,MR4290385,DGV-dilute,MR4280836}, derivation of higher order models \cite{MR4259909,feppon:hal-02518528,feppon:hal-03098222} and so on.
\subsection{Geometric set-up of the perforated domain}\label{sec:geosetup} The model set (of holes) is denoted by $T$, and it is an open subset of the unit cube $Q_1 := (-\frac12,\frac12)^d$. For each fixed $\eta \in (0,1]$, we define the perforated cube in the unit scale by $Y_{\rm{f},\eta} := Q_1\setminus \eta \ol T$, where the subscript ``f'' stands for ``fluid'' part; this is a convention frequently used for flows in porous media. We assume $\ol T \subset B_{\frac13}(0)$ so that $T$ is separated from $\partial Q_1$ and $Y_{{\rm f},\eta}$ is connected even when $\eta=1$. We take copies of $Y_{{\rm f},\eta}$ and glue them together to form a $1$-periodic perforated space $\R^d_{\rm{f},\eta}$, defined by
\begin{equation*}
 \R^d_{\rm{f},\eta} := \bigcup_{k\in \Z^d} (k+Y_{\rm{f},\eta}\cup \partial Q) = \R^d \setminus \bigcup_{k\in \Z^d} (k+\eta\ol T).
\end{equation*} 
For each $\eps \in (0,1]$, we rescale $\R^d_{\rm{f},\eta}$ by $\eps$ to obtain the $\eps$-periodic perforated space $\eps \R^d_{\rm{f},\eta}$; equivalently, this is obtained by removing from $\R^d$ the sets $T^{\eps,\eta}_k := \eps k + \eps\eta T$, $k\in \Z^d$. Note that $\eps\eta$ is the length scale of the removed holes, and $\eps$ is the typical distance between neighboring sets of holes. Finally, given an open bounded set $\Omega \subset \R^d$ and for each fixed $\eps$ and $\eta$, we define the bounded perforated domain $\Omega^{\eps,\eta}$ by $\Omega \,\cap\, (\eps \R^d_{\rm{f},\eta})$ with some modifications near $\partial \Omega$. Denote the interior holes completely contained in $\Omega$ by
\begin{equation*}
\bigcup_{k \in J_{\rm int}} \{\eps k + \eps\eta T\}, \; \text{with } \; J_{\rm int} := \{k \in \Z^d \,:\, \eps(k+ B_{2\eta}) \subset \Omega\}. 
\end{equation*}
Denote the boundary holes that are too close to $\partial \Omega$ (they may even intersect $\partial \Omega$) by
\begin{equation*}
\bigcup_{k \in J_{\rm bdr}} \{(\eps k + \eps\eta T)\cap \Omega\}, \; \text{with } \; J_{\rm bdr} := \{k\in \Z^d\,:\,\eps(k+B_{2\eta})\cap \partial \Omega \ne \emptyset\}.
\end{equation*}
The boundary of $\Omega^{\eps,\eta}$ hence contains two parts: the interior boundary
\begin{equation*}
\partial \Omega^{\eps,\eta}_{\rm int} : = \bigcup_{k\in J_{\rm int}} (\eps k + \eps \eta\partial T),
\end{equation*}
and the exterior boundary
\begin{equation*}
\partial \Omega^{\eps,\eta}_{\rm ext} = \partial \Omega^{\eps,\eta}\setminus \partial \Oint = (\partial \Omega \cap \eps\R^{d}_{{\rm f},\eta}) \cup (\bigcup_{k\in J_{\rm bdr}} \Omega\cap (\eps k + \eps\eta\partial T))
\end{equation*}
Concerning the smoothness of $\Omega^{\eps,\eta}$, we assume $\partial \Omega$ and $\partial T$ are of class $C^{2,\alpha}$ for some $\alpha \in (0,1)$. Even with this assumption, the cutting of $\Omega$ and $\eps \R^d_{{\rm f},\eta}$ may result in irregular boundaries near $\partial \Omega$. To avoid such pathological cases we modify those holes too close to $\partial \Omega$ by enlarging them so that the resulting $\partial \Omega^{\eps,\eta}_{\rm ext}$ remains Lipschitz. One way to do this is as follows: let $k \in J_{\rm bdr}$ be the index of such a boundary cell, and the neighborhood $B_{\eta\eps}(\eps k)$ of $T^{\eps,\eta}_k$ intersects $\partial \Omega$. Let $z$ denote the point in $\partial \Omega$ that is closest to $\eps k$ and, moreover, $\eps k = z + t N_z$ for some $t \in \R$. We enlarge the hole $T^{\eps,\eta}_k \cap \Omega$ by removing the segments
\begin{equation*}
\{x + r \,N_z \,:\, r> 0, x \in B_{\eps\eta}(\eps k) \cap \Omega\}.
\end{equation*}
The enlarged boundary holes are still denoted by $T^{\eps,\eta}_k$, $k\in J_{\rm bdr}$. As it will be clear later, the precise modification plays no important role. 

For notational simplicity, in most of the paper $\eta$ is omitted from the super- and subscriptions in $\Omega^{\eps,\eta}$, $T^{\eps,\eta}$, $Y_{\rm{f},\eta},\R^d_{\rm{f},\eta}$, $\eps\R^d_{\rm{f},\eta}$, $u^{\eps,\eta}$, etc.

We always assume that $\Omega^\eps$ is set up as above and, moreover, the following assumptions hold. In the rest of the paper, those conditions altogether are referred to as  the \emph{geometric set-up} (A):
\begin{enumerate}
  \item[(A1)] 
$T$ is the union of several disjoint star-shaped open sets, $\ol T\subset B_{1/3}$, $\partial T$ is of class $C^{2,\alpha}$. 
\item[(A2)]  $\Omega$ is simply connected, and $\partial \Omega$ is of class $C^{2,\alpha}$.
  \item[(A3)] Boundary holes are modified when necessary so that $\partial \Omega^\eps$ is Lipschitz (not uniform in $\eps,\eta$).
  \end{enumerate}
We remark that the assumption that $\Omega$ is simply connected and the precise bound of $T$ are not crucial and can be relaxed. Note also that by assumption $Y_{\rm f}$, $\R^d_{\rm f}$ and hence $\Omega^\eps$ are connected. 

\subsection{Background and the main results}

As seen from the geometric set-up, the parameter $\eta$ (which may depend on $\eps$) is the relative smallness of the holes with respect to the periodicity $\eps$. When $\eta$ is a fixed parameter in $(0,1)$ as $\eps$ goes to zero, we are in the classical setting and the effective model can be derived from a formal two-scale expansion. Plug in the ansatz
\begin{equation*}
u^\eps(x) = u_0(x,y) + \eps u_1(x,y) + \eps^2 u_2(x,y) + \dots, \quad  \text{with} \quad y = \tfrac{x}{\eps}
\end{equation*}
to \eqref{eq:hetnp}, apply the chain rule and replace the gradient $\nabla$ by $\nabla_x + \frac{1}{\eps} \nabla_y$. We get
\begin{equation*}
\left(\tfrac{1}{\eps^2} \mathcal{L}_2 + \tfrac{1}{\eps} \mathcal{L}_1 + \mathcal{L}_0\right) \left(u_0(x,y) + \eps u_1(x,y) + \eps^2 u_2(x,y) + \dots \right)  = f,
\end{equation*}
where the differential operators $\mathcal{L}_i$, $i=0,1,2$, are defined by
\begin{equation*}
\mathcal{L}_2 = -\Delta_y, \quad \mathcal{L}_1 = -2\nabla_y \cdot \nabla_x, \quad \mathcal{L}_0 = -\Delta_x.
\end{equation*}
Assume that each $u_i$ is defined on $\Omega\times (\bT^d\setminus \eta \ol T)$ where $\bT^d$ stands for the unit torus (i.e.\,periodic cell), in particular they are periodic in the $y$ variable. Making terms of the same order in $\eps$ equal, we obtain a hierarchy of equations in $y$ with boundary conditions that the $u_j$'s must satisfy. The leading order equations yield $u_0(x,y) = u_0(x)$ and the next order leads to
\begin{equation*}
u_1(x,y) = \sum_{k=1}^d \chi_k(y) \frac{\partial u_0}{\partial x_k}(x) + u_{1,0}(x)
\end{equation*}
where $u_{1,0}(x)$ is a constant function in $y$, and, for each $k=1,\dots,d$, the function $\chi_k = \chi_{k,\eta}$ is the unique (up to an additive constant) solution to the \emph{cell problem}
\begin{equation}
\label{eq:cellp}
\left\{
\begin{aligned}
&-\Delta \chi_{k,\eta}(y) = 0, \qquad &y \in \bT^d\setminus \eta\ol T,\\
&N_y\cdot(\nabla \chi_{k,\eta} + e_k)(y) = 0, \qquad &y \in \eta\partial T.
\end{aligned}
\right.
\end{equation}
Here, $e_k$'s  are the unit coordinate vectors and $N$ denotes the outer normal vector along $\partial (\eta T)$. The next order equations in the hierarchy produce the problem for $u_2$:
\begin{equation*}
\left\{
\begin{aligned}
&-\Delta_y u_2(x,y) = f(x) + \Delta_x u_0(x) + 2\partial_{y_\ell} \chi_k(y) \partial^2_{x_\ell x_k} u_0(x), &\quad &y \in \bT^d\setminus \ol T,\\
&N_y \cdot \nabla_y u_2(x,y) = -N_y^\ell \partial^2_{x_\ell x_k} u_0(x) \chi_k(y), &\quad &y \in \eta\partial T.
\end{aligned}
\right.
\end{equation*}
As always the summation convention is assumed. By a compatibility condition the integrals of the right hand sides, in $\bT^d\setminus \eta\ol T$ and on $\eta\partial T$ respectively, must be equal. This leads to
\begin{equation*}
(1-|\eta T|) f = - \left(\int_{\bT^d\setminus \eta\ol T} \delta_{\ell k} + \partial_\ell \chi_k(y)dy \right)\partial^2_{x_\ell x_k} u_0(x).
\end{equation*}
After some rewriting, we conclude that
\begin{equation}
\label{eq:hometa}
\left\{
\begin{aligned}
&-\nabla \left(\ol A(\eta) \nabla \ol u^\eta(x)\right) = f(x), &\quad &x\in \Omega,\\
&\ol u^\eta(x) = g(x), &\quad &x\in \partial \Omega.
\end{aligned}
\right.
\end{equation}
where the homogenized coefficients $\ol A(\eta) = (\ol a_{ij}^\eta)$ are \emph{constants} defined by
\begin{equation}
\label{eq:aijeta}
\ol a_{ij}^\eta = \fint_{\bT^d \setminus \eta\ol T} \delta_{ij} + \partial_i \chi_{j,\eta}(y) dy.
\end{equation}
Throughout the paper, $\fint$ stands for averaged integrals. The matrix $\ol A(\eta)$ is known to remain symmetric and elliptic. Rigorous results showing that $u^\eps$ converging in proper sense to $\ol u$ were proved e.g.\,in \cite{MR773850,MR1195131,AllMur} even in more general settings.

The dilute setting, i.e.\,when $\eta_\eps \to 0$ together with $\eps$, was studied by Conca and Donato in \cite{MR974289}. They showed that, with periodic Neumann data $h$ prescribed at the inner boundary, there is a critical scale $\eta_{{\rm cr},1}(\eps)$ for $\eta$ above which the energy $\|\nabla u^\eps\|^2_{L^2}$ blows up as $\eps \to 0$ if $h$ has non-zero mean over the boundary of the holes. This scale guarantees that $\lim_{\eps \to 0} \eps \eta^{-(d-1)}_\eps$ is in $(0,\infty)$. Note that $\eta_{{\rm cr},1} \sim \eps^{\frac{1}{d-1}}$ for $d\ge 2$, and at this scale the total surface volume of $\partial \Oint$ remains finite as $\eps \to 0$. It is natural to compare $\eta_{{\rm cr},1}$ with the critical scale $\eta_{{\rm cr},2}$ that occurs in homogenization of perforated domains with zero Dirichlet data set in the holes. The latter is determined so that $\sigma_\eps$ has limit in $(0,1)$, where 
\begin{equation}
\label{eq:sigepsdef}
\sigma_\eps = \begin{cases} \eps \eta^{-\frac{d-2}{2}}, \quad & d\ge 3,\\
\eps|\log \eta|^{\frac12}, \quad &d=2.
\end{cases}
\end{equation}
$\sigma_\eps$ is the bounding constant of a version of Poincar\'e inequality for $H^1$ functions in $B_\eps$ with zero value in $B_{\eta\eps}$; see e.g.\,\cite{Allaire91-2,MR4075336}. Note that $\eta_{{\rm cr},2} \sim \eps^{\frac{2}{d-2}}$ for $d\ge 3$ and $\eta_{{\rm cr},2} \sim \exp(-\frac{C}{\eps^2})$ for $d=2$. Clearly, $\eta_{{\rm cr},1}$ is of an order much larger than $\eta_{{\rm cr},2}$. 

For zero Neumann data, i.e.\,$h=0$ as in the second line of \eqref{eq:hetnp}, the dilute setting is very simple and the homogenized equation is
\begin{equation}
\label{eq:homnp}
\left\{
\begin{aligned}
&-\Delta u(x) = f(x), &\quad &x \in \Omega,\\
&u(x) = g(x), &\quad &x \in \partial \Omega.
\end{aligned}
\right.
\end{equation}
In other words, the effective conductivity matrix $\ol A(0)$ is simply the identity matrix $I$. Moreover, even without using correctors, it is easy to verify that $u^\eps$ converges \emph{strongly} in $H^1$ to $u$; see \cite{MR974289}. As far as we know, convergence rates for homogenization in the \emph{dilute} setting, even for zero Neumann data, have not been studied before. We also remark that cell problems \eqref{eq:cellp} was not used in \cite{MR974289}. A main contribution of this paper is to establish such results by analyzing the two-scale expansion and the cell problems. 

Our first main result concerns the continuous transition of effective models from \eqref{eq:hometa} to \eqref{eq:homnp} with respect to $\eta$ (understood as the limit of $\eta_\eps$). Let $\mathbb{M}^d$ denote the set of $d\times d$ symmetric and positive definite real matrices. Define $\ol A : [0,1) \to \mathbb{M}^d$ by \eqref{eq:aijeta} for $\eta \in (0,1)$ and by $I$ for $\eta = 0$.

\begin{theorem}\label{thm:Aeta} Under the geometric set-up conditions {\upshape(A)}, the mapping $\eta \mapsto \ol A(\eta)$ defined above is continuous, and moreover, the following expansion holds:
\begin{equation}
\label{eq:abarijexpan1}
\ol a_{ij}^\eta = \delta_{ij} - \eta^d \int_{\R^d \setminus \ol T} \nabla w^0_i \cdot \nabla w^0_j + O(\eta^{2d}), \qquad \text{as } \eta \to 0.
\end{equation}
In particular, there exists a constant $C>0$ independent of $\eta$ so that for sufficiently small $\eta$, we have $|\ol A(\eta) - I | \le C\eta^d$. Here, for each $k \in \{1,\dots,d\}$, $w^0_k$ is the unique solution to the exterior Neumann problem
\begin{equation}
\label{eq:extnp}
\left\{
\begin{aligned}
&-\Delta w_k(z) = 0, &\qquad &z\in \R^d \setminus \ol T,\\
&\frac{\partial w_k}{\partial N}(z) = -N_z, &\qquad &z \in \partial T,\\
&|w_k(z)| \to 0, &\qquad &|z| \to \infty,
\end{aligned}
\right.
\end{equation}
\end{theorem}

\begin{remark} 
The existence and uniqueness of $w^0_k$ is classical and a formula using single-layer potential operator is given in Section \ref{sec:rcellp}. We may view \eqref{eq:abarijexpan1} as an expansion of the effective diffusion in terms of the hole volume fraction $\gamma := \eta^d|T|$, as in \cite{MR2129229,DGV-dilute}; then it is a first order formula although the second order could also be read from the proofs in Section \ref{sec:rcellp}. Due to the definite sign of the second term, we see the effective diffusion is \emph{slowed} due to the collective effects of holes.
\end{remark}

The proof of the above theorem is in Section \ref{sec:Aexpan} and is based on explicit representation and detailed analysis in Section \ref{sec:rcellp} of the cell problems \eqref{eq:rcellp} rescaled from \eqref{eq:cellp}. Expansions of this type were previous obtained in \cite{MR2129229,MR2244590,MR3025042} for effective conductivity or elasticity coefficients of periodic two-phase materials. The basic tool is the periodic layer potential theory and we supply some key results in Section \ref{sec:lp}. To see what the above theorem implies for solutions of PDEs, for each fixed $\eps$ and $\eta$ (which depends on $\eps$), let $u^{\eps,\eta}$ be the solution to \eqref{eq:hetnp}. Treat $\eta$ as fixed and let $\ol A(\eta)$ be defined by \eqref{eq:aijeta}. Let $\ol u^\eta$ and $u$ be the solutions to \eqref{eq:hometa} and \eqref{eq:homnp} respectively. We can quantify the convergence of $\ol u^\eta$ to $u$, which shows the continuous transition of the effective models as the holes become dilute. 

\begin{corollary}\label{coro:uetau} Assume that the geometric set-up {\upshape(A)} holds. Then there is a constant $C> 0$ independent of $\eps$ and $\eta$, such that for any $\eta > 0$ sufficiently small, for any $f\in L^2(\Omega)$ and $g\in H^1(\Omega)$, we have
\begin{equation}
\label{eq:uetauH1}
\|\nabla (\ol u^\eta - u)\|_{L^2(\Omega)} \le C\eta^d(\|f\|_{L^2} + \|g\|_{H^1}).
\end{equation}
If we assume further $f\in C^\alpha(\ol \Omega)$ and $g\in C^{2,\alpha}(\ol \Omega)$, then
\begin{equation}
\label{eq:uetauC2}
\|\ol u^\eta -u\|_{C^{2,\alpha}(\ol\Omega)} \le C\eta^d(\|f\|_{C^\alpha} + \|g\|_{C^{2,\alpha}}).
\end{equation}
\end{corollary}

The above is a direct consequence of the quantitative behavior of $\ol A(\eta)$ in \eqref{eq:abarijexpan1} and the standard elliptic PDEs theory. We only briefly outline the proof in Remark \ref{rem:uetau} below.
Our next main result is the convergence rate for $u^\eps$ to $u$, as $\eps$ and $\eta$ go to zero. In view of Corollary \ref{coro:uetau}, we may quantify the difference between $u^{\eps}$ and $\ol u^\eta$ instead. We only treat the simplest setting for which $\ol u^\eta$ and $u$ are smooth solutions to \eqref{eq:hometa} and \eqref{eq:homnp}, to emphasize the role of $\eta$.

\begin{theorem}\label{thm:uetarate} Assume that the geometric set-up {\upshape(A)} holds. For each $k \in \{1,\dots,d\}$, let $\chi_{k,\eta}$ be solutions to the cell problems \eqref{eq:cellp}. Then there is a constant $C > 0$ such that, for all $\eta$ sufficiently small, for any $f\in C^\alpha(\Omega)$ and $g \in C^{2,\alpha}(\ol \Omega)$, the solution $u^{\eps}$ of \eqref{eq:hetnp} and $\ol u^\eta$ of \eqref{eq:hometa} satisfy
\begin{equation}
\label{eq:uetaH1}
\begin{aligned}
&\|u^{\eps} - \ol u^\eta - \eps \chi_{\ell,\eta}(\tfrac{x}{\eps})\partial_\ell \ol u^\eta\|_{H^1(\Omega^\eps)} \\
& \qquad \qquad  \qquad \le  C(\|f\|_{C^\alpha} + \|g\|_{C^{2,\alpha}})\times 
\begin{cases}
\eta^{d-1}, \;  &\text{if } \eta^{d-2}/\sigma^2_\eps \gg 1, \\
\sqrt{\eps}\eta^{\frac d2}, \; &\text{if } \eta^{d-2}/\sigma^2_\eps \lesssim 1, \, d\ge 3,\\
\sqrt{\eps}\eta|\log \eta|^{\frac{1}{2}}, \; &\text{if } \eta^{d-2}/\sigma^2_\eps \lesssim 1, \, d=2.
\end{cases}
\end{aligned}
\end{equation}
\end{theorem}
When $\eta$ is fixed and of order one, $\ol u^\eta$ is precisely the homogenized solution and convergence rates for the left hand side of \eqref{eq:uetaH1} was proved in \cite{MR1195131}, in the more general context of elasticity (elliptic) systems, and the bound is given by $C\sqrt{\eps}\|\ol u^\eta\|_{H^3}$. Their result was generalized recently in \cite{MR3659366} and in \cite{MR4240768} for less regular holes and less regular $\ol u^\eta$. The novelty of our result here is the explicit dependence on $\eta$. Theorem \ref{thm:uetarate} is proved in Section \ref{sec:uetarate}. We do not recover the classical rate $O(\sqrt{\eps})$ by setting $\eta = 1$ in \eqref{eq:uetaH1} as some of the estimates used in the proof are useful only in the dilute setting. Note also that the rate $O(\eta^{d-1})$ is sharper than $O(\sqrt{\eps})$ when $\eta \lesssim \eps^{\frac{1}{2(d-1)}}$, that is, below the scale $\sqrt{\eta_{{\rm cr},1}}$.

The term $z_\eps := \eps \chi_{k,\eta}\partial_k \ol u^\eta$ is the so-called \emph{corrector}. When $\eta$ is of order $O(1)$, it is necessary to add $z_\eps$ to $\ol u^\eta$ so that the corrected quantity $\ol u^\eta + z_\eps$ is close to $u^\eps$ in $H^1$. As mentioned earlier, in the dilute setting with zero Neumann data in the holes, such correctors are not needed for the strong convergence of $u^\eps \to u$ in $H^1$. A simple weak convergence argument suffices; see \cite{MR974289}. Moreover, a quantification of the argument (see Remark \ref{rem:altmethod}) yields
\begin{equation}
\label{eq:uepsuH1}
\|u^\eps - u\|_{H^1(\Omega^\eps)} \le C\eta^{\frac d2}.
\end{equation}
Using the results in Section \ref{sec:rcellp} we check $\|z_\eps\|_{H^1(\Omega)} \le C\eta^{\frac d2}$. This together with \eqref{eq:uetaH1} also yield the above estimate. Apparently, the one in Remark \ref{rem:altmethod} is easier. On the other hand, for $d\ge 3$ the right hand side in \eqref{eq:uetaH1} is much smaller than $O(\eta^{\frac d2})$; hence, for $d\ge 3$, adding the corrector to $\ol u^\eta$ yields a sharper approximation of $u^\eps$ in $H^1$.

At this point we also mention that the boundary modification of $\Omega^\eps$ near $\partial \Omega$ in Section \ref{sec:geosetup} is more complicated than a more commonly used one, namely removing (or filling in) the boundary holes, as in  \cite{MR3659366,MR4240768}. The exterior boundary in this latter case is simply $\partial \Omega$. We elaborate the apparently more complex case to demonstrate a fact taken often as granted: boundary modifications do not play essential roles in homogenization. 

\subsection*{Notations} The unit (flat) torus of $\R^d$, i.e.\,the closed unit cube $[-\frac12,\frac12]^d$ with opposite faces identified, is denoted by $\bT^d$. The open unit cube $(-\frac12,\frac12)^d$ is denoted by $Q_1$. Functions in $\bT^d$ are treated at the same time as $\bT^d$-periodic functions defined in $\R^d$. Similarly, functions in $\bT^d\setminus \eta \ol T$ are treated as $\bT^d$-periodic functions defined in $\R^d_{\rm f}$. Given a set $A \subset \R^d$ and $r > 0$, $rA$ denotes the rescaled set $\{rx \,:\, x \in A\}$, $\fint_A$ is the averaged integral over $A$, $|A|$ is its volume and $|\partial A|$ is the surface volume of its boundary $\partial A$. 

Under the geometric set-up (A), the space of $H^1(\Omega^\eps)$ functions vanishing in $\partial \Oext$ is denoted by
\begin{equation}
\label{eq:Vepsdef}
V_\eps := \{v \in H^1(\Omega^\eps) \,:\, v = 0 \, \text{ in } \, \partial \Oext\}.
\end{equation}
We use $L^2_0(\partial T)$ to denote the space of square integrable functions in $\partial T$ that are also mean zero. Concerning Neumann boundary conditions at the boundary of the holes, $N$ always denotes the outer normal vector pointing out of the hole. Throughout the paper, summation convention that repeated indices are summed over is always invoked. Finally, a bounding constant $C$ in an estimate is called universal if it depends only on $d,\alpha,\Omega,T$.

\section{Preliminary results}

In this section, we give some basic results that will be used to prove the main theorems. For each fixed $\eps,\eta \in (0,1)$, let $\Omega^\eps$ be defined as in Section \ref{sec:geosetup}. We recall an energy estimate for Poisson equations in $\Omega^\eps$ with general Neumann data prescribed at the inner boundary. This will be used to quantify the convergence rates in Theorem \ref{thm:uetarate}. We also introduce the periodic layer potential operators and some fundamental properties, which are key tools for the analysis of the rescaled cell problems.

\subsection{Basic energy estimates}
\label{sec:energy}
We consider the mixed boundary value problem:
\begin{equation}
\label{eq:mbvppd}
\left\{
\begin{aligned}
&-\Delta w = \div F + f_0, &\quad &\text{in } \Omega^\eps,\\
&\frac{\partial w}{\partial N} = -N\cdot F + h, &\quad &\text{in } \partial \Omega^\eps_{\rm int},\\
&w = g_0, &\quad &\text{in } \partial \Omega^\eps_{\rm ext}.
\end{aligned}
\right.
\end{equation}
Here, $f_0$ and $F = (F^i)$ are, respectively, scalar and vector fields belonging to $L^2(\Omega)$, $g_0 \in H^1(\Omega)$ and $h \in L^2(\partial \Omega^\eps_{\rm int})$. By a weak solution to \eqref{eq:mbvppd} we mean a function $w \in H^1(\Omega^\eps)$ satisfying $w - g_0 = 0$ in $\partial \Omega^\eps_{\rm ext}$ and
\begin{equation*}
\int_{\Omega^\eps} \nabla w\cdot \nabla \varphi = \int_{\Omega^\eps} f_0 \varphi - F \cdot \nabla \varphi + \int_{\partial \Omega^\eps_{\rm int}} h \varphi, \qquad \forall \varphi \in V_\eps.
\end{equation*}
Evidently, the identity above can be replaced by
\begin{equation}
\label{eq:wfmbvp}
\int_{\Omega^\eps} \nabla (w-g_0)\cdot \nabla \varphi = \int_{\Omega^\eps} f_0 \varphi - (\nabla g_0+ F) \cdot \nabla \varphi + \int_{\partial \Omega^\eps_{\rm int}} h \varphi. 
\end{equation}
Thanks to the Poincar\'e type inequality \eqref{eq:pipd} which can be applied to $w-g_0$, an application of the Lax-Milgram theorem shows, for each fixed $\eps$ and $\eta$, there is a unique weak solution $w \in H^1(\Omega^\eps)$ for \eqref{eq:mbvppd}. Define the parameter
\begin{equation}
\label{eq:kappadef}
\kappa_{\eps,\eta} := \begin{cases}
\max\{(\eps^{-1}\eta^{d-1})^{\frac12}, (\eps\eta)^{\frac12}\},  \quad &\text{if } d\ge 3,\\
\max\{(\eps^{-1}\eta^{d-1})^{\frac12}, (\eps\eta|\log \eta|)^{\frac12}\}, \quad &\text{if } d=2.
\end{cases}
\end{equation}
The following energy estimate holds.

\begin{theorem}\label{thm:energyest} Assume the geometric set-up {\upshape(A)}. There is a universal constant $C > 0$ such that, for any fixed $\eps$ and $\eta$, the unique weak solution $w$ to the problem \eqref{eq:mbvppd} satisfies
\begin{equation}
\label{eq:eembvp}
\|w\|_{H^1(\Omega^\eps)} \le C\{\|f_0\|_{L^2(\Omega^\eps)} + \|F\|_{L^2(\Omega^\eps)} + \|g_0\|_{H^1(\Omega^\eps)} + \kappa_{\eps,\eta} \|h\|_{L^2(\partial \Omega^\eps_{\rm int})}\}.
\end{equation}
\end{theorem} 

We note that $\kappa_{\eps,\eta}$ blows up as $\eps\to 0$ if $\eta = \eta_\eps$ shrinks slower than $\eps^{\frac{1}{d-1}}$, i.e.\,if $\eta \gg \eta_{{\rm cr},1}$. Secondly, in both lines of \eqref{eq:kappadef}, the first terms dominate if $\eta \gg \eta_{{\rm cr},2}$ and the second terms dominate if $\eta \lesssim \eta_{{\rm cr},2}$; see the paragraph above and below \eqref{eq:sigepsdef} for the definitions of $\eta_{{\rm cr},1}$ and $\eta_{{\rm cr},2}$ respectively.

\begin{proof} We need to estimate the trace in $\partial \Oint$ for functions in $H^1(\Omega^\eps)$; this is given in Proposition \ref{prop:trace} and is due to Conca and Donato \cite{MR974289}. The rest of the proof is standard. Take $\varphi = w-g_0$ in \eqref{eq:wfmbvp}; we get
\begin{equation*}
\begin{aligned}
\|\nabla(w-g_0)\|_{L^2(\Omega^\eps)}^2 \le \; &\|f_0\|_{L^2(\Omega^\eps)}\|w-g_0\|_{L^2(\Omega^\eps)} + (\|\nabla g_0 + F\|_{L^2(\Omega^\eps)})\|\nabla(w-g_0)\|_{L^2(\Omega^\eps)}\\
 &\; + \|h\|_{L^2(\partial \Oint)}\|w-g_0\|_{L^2(\partial \Oint)}.
 \end{aligned}
\end{equation*}
Apply the trace inequality \eqref{eq:trimbvp} to $w-g_0$ for the last item and then the Poincar\'e inequality \eqref{eq:pipd} for $\|w-g_0\|_{L^2(\Omega^\eps)}$. We then can control $\|w-g_0\|_{H^1}$ and $\|w\|_{H^1(\Omega^\eps)}$ as desired.
\end{proof}

\subsection{The periodic layer potentials}\label{sec:lp}

Recall that $\Gamma \in C^\infty(\R^d\setminus \{0\})$ defined by
\begin{equation*}
\Gamma(x) = \begin{cases}
\frac{1}{(2-d)|\partial B_1|}\frac{1}{|x|^{d-2}}, \quad &d\ge 3,\\
\frac{1}{2\pi} \log |x| \quad &d=2,
\end{cases}
\end{equation*}
is the fundamental solution of the Laplace equation in $\R^d$, i.e.\,$\Delta \Gamma = \delta_0$. Classical layer potential theory provides useful tools to solve Laplace equations inside and exterior to $T$, with prescribed Dirichlet, Neumann or transmission data; see e.g.\,\cite{Folland,AmmKan}.

In the next section we will study the rescaled cell problem \eqref{eq:rcellp}, which is a Laplace equation in the punctured torus $\frac1\eta\bT^d\setminus \ol T$, with Neumann data at the boundary $\partial T$. The periodic variant of layer potentials can be used to solve it. For this, we consider the fundamental solution $G$ on the unit torus, i.e.
\begin{equation*}
  \Delta G = \delta_0 - 1 \quad \text{in } \bT^d, \qquad \textstyle\int_{\bT^d} G = 0.
\end{equation*}
It is known that $G = \Gamma + R$, and $R$ is a smooth function in the unit cube $Q_1$. More precisely, $R$ is the unique solution that solves
\begin{equation}
\label{eq:Rdef}
  -\Delta R = 1 \quad \text{in } Q_1, \qquad R + \Gamma(x) \, \text{ is in } C^\infty(\bT^d\setminus \{0\}).
\end{equation}
Note that $R$ is smooth in $\ol Q_1$ but derivatives of $R$ do not match on opposite sides of $\partial Q_1$.

For $d\ge 2$, we can check that $G^\eta(x) := \eta^{d-2}G(\eta x)$ is the fundamental solution of the Laplace equation in the rescaled torus $\tfrac{1}{\eta} \bT^d$, i.e.
\begin{equation*}
\Delta G^\eta = \delta_0 - \eta^d \quad \text{in } \textstyle\frac{1}{\eta}\bT^d, \qquad \textstyle\int_{\frac{1}{\eta} \bT^d} G^\eta = 0.
\end{equation*}
We then verify that $G^\eta$ is a simple perturbation of $\Gamma$, as follows
\begin{equation}
\label{eq:Gpert}
\begin{aligned}
G^\eta(x) &= \Gamma(x) + \begin{cases}
\eta^{d-2} R(\eta x), \quad &d\ge 3,\\
\frac{1}{2\pi} \log \eta + R(\eta x), \quad &d=2.
\end{cases}\\
\nabla G^\eta(x) &= \nabla \Gamma(x) + \eta^{d-1} \nabla R(\eta x), \qquad\quad\;\; d\ge 2.
\end{aligned}
\end{equation}

The periodic single-layer and double-layer potential operators associated to the set $T$ (viewed as a subset of $\frac1\eta\bT^d$) are then defined by replacing $\Gamma$ by $G^\eta$ in the classical operators: given a momentum function $\phi$ defined in $\partial T$ and for $x\in \frac1\eta \bT^d \setminus \partial T$, they are defined, respectively, by
\begin{eqnarray}
  \cS^\eta_\fp [\phi](x) &:=& \int_{\partial T} G^\eta(y-x) \phi(y) dy,\label{eq:cSeta}\\
  \cD^\eta_\fp [\phi](x) &:=& \int_{\partial T} \frac{\partial G^\eta(y-x)}{\partial N_y} \phi(y) dy.\label{eq:cDeta}
\end{eqnarray}
We mainly put the operators in the $L^2$ setting. It is easy to check the following: for $\phi \in L^2(\partial T)$, the double-layer potential $\cD^\eta_\fp[\phi]$ always defines a harmonic function in $\frac1\eta \bT^d \setminus \partial T$, but the single-layer potential $\cS^\eta_\fp[\phi]$ is harmonic there if and only if $\int_{\partial T} \phi = 0$. On the other hand, $\cS^\eta_\fp[\phi]$ is continuous across $\partial T$ and the trace there is still denoted by $\cS^\eta_\fp[\phi]$ as the definition \eqref{eq:cSeta} makes sense also for $x\in \partial T$. $\cS^\eta_\fp[\phi]$ hence defines a continuous function over $\eta^{-1}\bT^d$. 

The most important property of layer potential operators are the trace formulas, also known as jump relations. A simple modification of the classical theory yields: for $x\in \partial T$,
\begin{equation}
\label{eq:SDtrace}
  \begin{aligned}
  \frac{\partial \cS^\eta_\fp[\phi]}{\partial N}\Big\rvert_\pm (x) &:= \lim_{t\to 0+} N_x \cdot \left(\nabla \cS^\eta_\fp[\phi]\right)(x\pm tN_x) = (\pm \frac12 I + \cK^{\eta,*}_\fp)[\phi](x),\\
  \cD^\eta_\fp[\phi]\big\rvert_\pm(x) &:= \lim_{t\to 0+} \cD^\eta_\fp[\phi](x\pm tN_x)= (\mp \frac12 I + \cK^\eta_\fp)[\phi](x).
\end{aligned}
\end{equation}
Here, $\cK^\eta_\fp$ is defined by the same formula as $\cD^\eta_\fp$ in \eqref{eq:cDeta} but now both variables $x,y$ live in $\partial T$, i.e.,
\begin{equation}
\label{eq:cKeta}
\cK^\eta_\fp [\phi](x) := \int_{\partial T} \frac{\partial G^\eta(y-x)}{\partial N_y} \phi(y) dy, \qquad x\in \partial T,
\end{equation}
and $\cK^{\eta,*}_\fp$ is the adjoint of $\cK^\eta_\fp$. In view of the regularity of $\partial T$ (being $C^{1,\alpha}$ is sufficient), they are regular integral operators on $\partial T$, and, moreover, compact linear transformations of $L^2(\partial T)$. The jump relations above show that $\cD^\eta_\fp[\phi]$ and $N_x\cdot \nabla\cS^\eta_\fp[\phi]$ have jumps across $\partial T$. They also suggest how to solve boundary value problems in $T$ and $\ol T^c$ with Dirichlet or Neumann data put in $\partial T$.

\subsubsection{Solvability for exterior Neumann problems}

In \cite{MR4075336,MR4172687}, we have systematically studied the properties of $-\frac12 I + \cK^{\eta,*}_{\fp}$ and its adjoint operator $-\frac12 I + \cK^\eta_{\fp}$ which provide key tools for the asymptotic analysis of the exterior Dirichlet problems in the punctured torus $\tfrac1\eta \bT^d\setminus \ol T$. For the exterior Neumann problem \eqref{eq:rcellp}, we need to study $\frac12 I + \cK^\eta_{\fp}$ and $\frac12 I + \cK^{\eta,*}_{\fp}$. 

For notational simplicity, we present the case of $\eta = 1$ and omit $\eta$ from the notations. 
Recall (see \cite{MR4075336}) that the constant function with value $1$ over $\partial T$ satisfies
\begin{equation}
\label{eq:-hK1=vT}
(-\frac12 I + \cK_{\fp})[1] = -|T|.
\end{equation}
If $h \in \ker(\frac12 I + \cK^*_{\fp})$, we have
\begin{equation*}
\int_{\partial T} h = -\frac{1}{|T|}\int_{\partial T} h(-\frac12 I + \cK_{\fp})[1]  = -\frac{1}{|T|}\int (-\frac12 I + \cK^*_{\fp})[h] = |T|^{-1}\int_{\partial T} h.
\end{equation*}
Since by assumption $|T|\ne 1$, we conclude that $h \in L^2_0(\partial T)$. This shows $\ker(\frac12 I +\cK^*_{\fp}) \subset L^2_0(\partial T)$, and if we set $u = \cS_{\fp}[h]$ and let $u_+$ and $u_-$ be its component outside and inside $T$, respectively, then $u_\pm$ is harmonic. Since $\frac{\partial u_+}{\partial N} = N\cdot \nabla u_+ = (\frac12 I + \cK^*_{\fp})[h] = 0$, the Green's identity
\begin{equation*}
0 = -\int_{\partial T} u_+ \frac{\partial u_+}{\partial N} = \int_{\bT^d \setminus \ol T} |\nabla u_+|^2 
\end{equation*}
shows that $u_+$ is a constant. By the continuity of $u$ across $\partial T$, we conclude that $u_-$ is the same constant in $T$, and, finally, that $h = \frac{\partial u_+}{\partial N} - \frac{\partial u_-}{\partial N} = 0$. This and the Fredholm theory allow us to conclude that
\begin{equation}
\label{eq:plpkerran}
\begin{aligned}
&\ker\left(\frac12 I + \cK_{\fp}^*\right) = \ker\left(\frac12 I + \cK_{\fp}\right) = \{0\}, \\
&\ran\left(\frac12 I + \cK_{\fp}^*\right) = \ran\left(\frac12 I + \cK_{\fp}\right) = L^2(\partial T). 
\end{aligned}
\end{equation}
Clearly, the argument above works also in the rescaled torus $\frac1\eta \bT^d$, and the above identities hold for $\cK^\eta_\fp$ and $\cK^{\eta,*}_\fp$. In particular, $\frac12 I + \cK^{\eta,*}_\fp$ is invertible in $L^2(\partial T)$.

\begin{remark}
Note that the connectedness of $\ol T^c$ is required above to conclude $\nabla u_+ = 0$, but $T$ is allowed to have the form $\cup_{i=1}^M D_i$ if $D_i$'s are disjoint and each $D_i$ is simply connected. 
Indeed, \eqref{eq:-hK1=vT} still holds in this case: the double-layer potential $\cD_{\fp,i}$ associated to each set $D_i$ is defined, and, for $x \in \bT^d\setminus (\cup_{i=1}^M D_i)$,
\begin{equation*}
\mathcal{D}_{\fp}[1](x) = \sum_{i=1}^M \mathcal{D}_{\fp}[\mathbf{1}_{\partial T_i}](x) = \sum_{i=1}^M \mathcal{D}_{\fp,i}[1](x) = -\sum_{i=1}^M |D_i| = -|T|.
\end{equation*}
This is a big contrast with the case of Dirichlet boundary conditions, as the dimension of $\ker(-\frac12 I + \cK_\fp)$ depends on the number $M$.
\end{remark}

\subsubsection{Perturbative results of periodic layer potentials} In view of the perturbative relations \eqref{eq:Gpert}, we expect that, for sufficiently small $\eta$, the operators $\cS^\eta_\fp$, $\cD^\eta_\fp$, $\cK^\eta_\fp$ and $\cK^{\eta,*}_\fp$ are perturbations of the corresponding classical layer potential operators $\cS, \cD, \cK$ and $\cK^*$. More precisely, $\cS$, $\cD$ are $\cK$ are defined as in \eqref{eq:cSeta}, \eqref{eq:cDeta} and \eqref{eq:cKeta}, but with $G^\eta$ replaced by $\Gamma$, and $\cK^*$ is the adjoint of $\cK$.
Indeed, for the single-layer potential, we observe that for $d\ge 3$,
\begin{equation}
\label{eq:cSpert}
  \cS^\eta_\fp = \cS + \eta^{d-2} \cR^\eta_1,
\end{equation}
where $\cR^\eta_1$ is a regular integral operator
\begin{equation*}
  \cR^\eta_1[\phi](x) = \int_{\partial T} R(\eta(y-x)) \phi(y) dy, \quad x\in \textstyle\frac1\eta Q_1.
\end{equation*}
Since $R\in C^\infty(\ol Q_1)$, the above defines a continuous function in $\frac1\eta\bT^d$. 
Similarly, for $d\ge 2$,
\begin{equation*}
\cK^{*,\eta}_{\fp} = \cK^* + \eta^{d-1}\cR^{\eta}_{2},
\end{equation*}
where $\cK^*$ is the Neumann-Poincar\'e operator associated to the free-space layer potentials associated to $\partial T$, and $\cR^{\eta}_{2}$ is a regular integral operator defined for $x\in \frac1\eta Q_1$. They are given, respectively, by
\begin{equation*}
\cK^*[h](x) := \int_{\partial T} \frac{1}{|\partial B_1|} \frac{N_x\cdot (y-x)}{|y-x|^{d}}h(y) dy, \quad \cR^{\eta}_{2}[h](x) := \int_{\partial T} N_x\cdot (\nabla R)(\eta(x-y)) h(y) dy.
\end{equation*}
It is well known that as long as $\R^d\setminus \ol T$ is connected, the operator $\frac12 I + \cK^*$ is invertible in $L^2(\partial T)$; see \cite[Chap.\,3]{Folland}; actually, the proof goes like the argument above \eqref{eq:plpkerran}. Recast $\frac12 I + \cK^{*,\eta}_\fp$ as
\begin{equation*}
(\tfrac12 I + \cK^*)\left[ I + \eta^{d-1} \cR^\eta_3\right], \quad \text{with} \quad \cR^\eta_3 = \left(\tfrac12 I + \cK^*\right)^{-1}\cR^\eta_2.
\end{equation*}
By \eqref{eq:plpkerran}, we already know that $\frac12 I + \cK^{*,\eta}_\fp$ is invertible in $L^2(\partial T)$. Using asymptotic analysis, we also obtain the following representation for it.

\begin{proposition}\label{prop:KpsNS} Under the geometric set-up {\upshape(A)}, and for $\eta \in(0,1)$ that is sufficiently small,
\begin{equation}
\label{eq:KstarNeumann}
\left(\frac12 I + \cK^{*,\eta}_\fp\right)^{-1} =\, \sum_{\ell =0}^\infty (-1)^\ell \eta^{\ell(d-1)} (\cR^\eta_3)^\ell \left(\frac12 I + \cK^*\right)^{-1}
\end{equation}
holds in $L^2(\partial T)$. The Neumann series above converges absolutely in $L^2(\partial T)$, and, moreover, it converges also in $C(\partial T)$, the space of continuous functions in $\partial T$.
\end{proposition}

\begin{proof} We already know that $(\frac12 I + \cK^*)^{-1}$ is bounded as a transformation in $L^2(\partial T)$. Since $\partial T$ is regular ($C^2$ is sufficient), one can show (c.f.\,\cite[Sec.\,3.B]{Folland}) that $(\frac12 I + \cK^*)$ is an invertible bounded linear transformation in $C(\partial T)$ as well: $\cK^*$ maps $L^\infty(\partial T)$ to $C(\partial T)$ and is compact as a linear transformation in $L^\infty(\partial T)$. Now given $g \in C(\partial T) \subset L^2(\partial T)$, there exists a unique $\phi \in L^2(\partial T)$ so that $(\frac12 I + \cK^*)[\phi] = g$, then $\phi \in L^2$ and $(\frac12 I + \cK^*)[\phi] \in C(\partial T)$ implies that $\phi \in C(\partial T)$. 
This shows the surjectivity in $C(\partial T)$ of $\frac12 I + \cK^*$. The injectivity is clear since $C(\partial T) \subset L^2(\partial T)$. The continuity of $\frac12 I + \cK^*$ is clear since $\cK^*$ is bounded as an operator in $L^\infty(\partial T)$, and the continuity of $(\frac12 I + \cK^*)^{-1}$ in $C(\partial T)$ then follows.

To prove the proposition, it suffices to show that $\cR^\eta_3$ is a bounded linear transformation in $L^2(\partial T)$ and in $C(\partial T)$, with operator norms uniformly bounded in $\eta$. Since it is easy to verify that $\cR^\eta_2$, and hence $\cR^\eta_3$, maps $L^2(\partial T)$ to $C(\partial T)$, it suffices to find a universal $C > 0$, so that $\|\cR^\eta_3\|_{L^2\to L^\infty} \le C$. To this end, we check that, for any $\phi \in L^2(\partial T)$,
\begin{equation*}
\|\cR^\eta_3[\phi]\|_{L^\infty} \le \|(\tfrac12 I + \cK^*)^{-1}\|_{L^\infty\to L^\infty} \|\cR^\eta_2[\phi]\|_{L^\infty} \le \|(\tfrac12 I + \cK^*)^{-1}\|_{L^\infty\to L^\infty} \|R\|_{C^1(Q_1)} |\partial T|^{\frac12} \|\phi\|_{L^2}.
\end{equation*}
Let $C$ be the constant above in front of $\|\phi\|_{L^2}$. For $\eta$ sufficiently small, namely $\eta^{d-1} < \frac34 C$, the Neumann series in the proposition converges absolutely, and the proof is complete.
\end{proof}

\begin{remark}\label{eq:NSL20inv} It is easy to check that $(\frac12I + \cK^*)$ and $(\frac12 I + \cK^*)^{-1}$ preserves the mean-zero property, i.e.\,$L^2_0(\partial T)$ is invariant under them. Similarly, $\cR^\eta_2$ (and hence $\cR^\eta_3$) also have this property: indeed, if $h \in L^2_0(\partial T)$, then
\begin{equation*}
\begin{aligned}
\int_{\partial T} \cR^\eta_2[h](x) dx &= \int_{\partial T}\left(\int_{\partial T} N_x \cdot (\nabla R)(\eta(x-y) dx\right) h(y) dy\\
&= \int_{\partial T} \left(\eta\int_T (\Delta R)(\eta(x-y)) dx \right) h(y) dy = -|T|\int_{\partial T} h = 0.
\end{aligned}
\end{equation*}
The terms in the Neumann series \eqref{eq:KstarNeumann}, hence, preserve $L^2_0(\partial T)$ also.
\end{remark} 

It is possible to obtain expansion formulas for the function $R$; see e.g.\,\cite{AmmKan,MR2244590}. Using such finer information, one can obtain more useful properties of $\cR^\eta_i$'s. 

\begin{proposition}\label{prop:Rexpan}
Let $d\ge 2$ and assume that the geometric set-up {\upshape (A)} holds.  Then the following results are true.
\begin{enumerate}
	\item[(1)] In a neighborhood of the origin, the function $R$ in  \eqref{eq:Rdef} behaves like
	\begin{equation}
	\label{eq:Rexpan}
	R(x) = R(0) - \frac1{2d}|x|^2 + O(|x|^4). 
	\end{equation}
	\item[(2)] Viewed as bounded linear operators from $L^2_0(\partial T)$ to $C(\partial T)$, $\cR^\eta_1$ and $\cR^\eta_3$ satisfy
	\begin{equation}
	\label{eq:cR1exp}
	\cR^\eta_1 = \eta^2 \cQ_1 + O(\eta^4), \qquad	\cR^\eta_3 = \eta (\frac12 I + \cK^*)^{-1} \cQ_2 + O(\eta^3), 	
	\end{equation}
	where 
	\begin{equation*}
\cQ_1[\phi](x) := -\frac1{2d}\int_{\partial T} |x-y|^2\phi(y) dy, \quad \cQ_2[\phi](x) := -\frac1{d}\int_{\partial T} N_x \cdot (x-y)\phi(y) dy.
	\end{equation*}
\end{enumerate}
\end{proposition}
\begin{proof}
The expansion formula of $R$ near the origin was obtained by Ammari and Kang \cite[Section 2.8]{AmmKan} using Fourier series representation of $G$; the qualitative feature of $R$, namely the evenness and the vanishing of first order derivatives, can be obtained using symmetry arguments. The expansion of $\cR^\eta_1$ then follows immediately; in particular, the constant $R(0)$ has no contribution if $\phi \in L^2_0(\partial T)$.
\end{proof}

\begin{remark} For $d=2$, in view of the formula \eqref{eq:Gpert}, the relation \eqref{eq:cSpert} still holds if the operators are restricted to the space $L^2_0(\partial T)$. The proposition further shows, for all $d\ge 2$ and $\phi \in L^2_0(\partial T)$,
\begin{equation}
\label{eq:cSpert1}
\cS^\eta_\fp[\phi] = \cS[\phi] + \eta^d \cQ_1[\phi] + O(\eta^{d+2}) \qquad \text{in } L^\infty(\partial T). 
\end{equation}
The $O(\eta^{d+2})$ means that this term has an operator norm from $L^2_0(\partial T)$ to $L^\infty(\partial T)$ of order $O(\eta^{d+2})$.
\end{remark}
\section{Analysis for the rescaled cell problems}
\label{sec:rcellp}


In this section we study the cell problem \eqref{eq:cellp} in the dilute setting, i.e.\,$\eta$ is small, and explore the asymptotic behaviors of the cell problem solutions as $\eps \to 0$. To make the removed hole fixed, we rescale the spatial variable $y \in \bT^d$ to $z = y/\eta$ and define the rescaled function
\begin{equation}
\label{eq:rschi}
{\tilde\chi}^{\,\eta}_k(z) := \textstyle\frac1{\eta}\chi_{k,\eta}(\eta z), \qquad z \in \textstyle\frac1\eta \bT^d \setminus \ol T.
\end{equation}
Then 
${\tilde\chi}^{\,\eta}_k$ solves
\begin{equation}
\label{eq:rcellp}
\left\{
\begin{aligned}
&-\Delta {\tilde\chi}^{\,\eta}_k(z) = 0, &\quad &z\in (\textstyle\frac1\eta \bT^d)\setminus \ol T,\\
&N_z\cdot \nabla {\tilde\chi}^{\,\eta}_k(z) = -e_k \cdot N_z, &\quad &z \in \partial T.
\end{aligned}
\right.
\end{equation}
We denote $e_k\cdot N$ by $N^k$ below. The above is an exterior Neumann problem in $\frac1\eta\bT^d$ outside the set $T$. We use layer potential theory to solve 
${\tilde\chi}^{\,\eta}_k$ and investigate its behavior from this rather explicit representation. In fact, ${\tilde\chi}^{\,\eta}_k$ can be written as a single-layer potential, and, hence, it is defined in the whole torus $\frac1\eta \bT^d$. We prove the following results for ${\tilde\chi}^{\,\eta}_k$ and $\chi_{k,\eta}$.

\begin{lemma}\label{lem:cprcp}
There exist universal positive constants $C$'s such that, for each $k\in \{1,\dots,d\}$ and for sufficiently small $\eta\in (0,1]$, the following holds.
\begin{itemize}
\item[(1)] The solution ${\tilde\chi}^{\,\eta}_k$ to \eqref{eq:rcellp} (and extended to $\frac1\eta\bT^d$ by the single-layer potential) satisfies
\begin{equation}
\label{eq:chietakH1bdd}
\|\nabla {\tilde\chi}^{\,\eta}_k\|_{L^2(\frac1\eta\bT^d)} + \eta \|{\tilde\chi}^{\,\eta}_k\|_{L^2(\frac1\eta \bT^d)} \le C.
\end{equation}
\item[(2)] The solution $\chi_{k,\eta}$ to \eqref{eq:cellp} is in $W^{1,\infty}(\bT^d\setminus \eta \ol T)$ and satisfies \begin{equation}
\label{eq:chietakbdd}
\|\chi_{k,\eta}\|_{L^\infty} \le C\eta  \quad\text{and}\quad \|\nabla \chi_{k,\eta}\|_{L^\infty} \le C.
\end{equation}
\item[(3)] For each $k=1,\dots,d$, let $w^0_k$ be defined as in \eqref{eq:extnp}. Then the following expansion holds:
\begin{equation}
\label{eq:chietakexpan}
\chi_{k,\eta} = \eta w^0_k + O(\eta^{d+1}), \qquad  \text{in } C(\eta\partial T).
\end{equation}
\end{itemize}
\end{lemma}

\begin{proof}
Fix a $k\in \{1,\dots,d\}$ below. Let ${\tilde\chi}^{\,\eta}_k$ be the solution to \eqref{eq:rcellp}. Since $\chi_{k,\eta} = \eta {\tilde\chi}^{\,\eta}_k(\cdot/\eta)$, the estimates \eqref{eq:chietakbdd} are equivalent to 
\begin{equation}
\label{eq:chietakbdd1}
\|{\tilde\chi}^{\,\eta}_k\|_{L^\infty(\frac1\eta \bT^d\setminus \ol T)} + \|{\tilde\chi}^{\,\eta}_k\|_{L^\infty(\frac1\eta \bT^d\setminus \ol T)} \le C.
\end{equation}
We seek for ${\tilde\chi}^{\,\eta}_k = \cS^\eta_\fp[\phi_k]$ and, by the trace formula \eqref{eq:SDtrace}, $\phi_k = \phi_{k,\eta}$ should solve
\begin{equation*}
(\tfrac12 I + \cK^{\eta,*}_{\fp})[\phi_{k,\eta}] = -N^k.
\end{equation*}
By \eqref{eq:plpkerran} this integral equation over $\partial T$ is uniquely solvable; moreover, in view of Proposition \ref{prop:KpsNS} and the formula \eqref{eq:KstarNeumann}, for $\eta$ sufficiently small, we have
\begin{equation}
\label{eq:phikexpan}
\phi_{k,\eta}= \sum_{\ell=0}^\infty (-1)^\ell \eta^{\ell(d-1)} (\cR^\eta_3)^\ell[\phi^0_k] = \phi^0_k - \eta^{d-1} \cR^\eta_3[\phi^0_k] + O(\eta^{2(d-1)}),
\end{equation}
where $\phi^0_k$ (independent of $\eta$) is the moment function $(\frac12 I + \cK^*)^{-1}[-N^k]$. Since $N^k \in C(\partial T)$, the above identity holds in $C(\partial T)$ and in $L^2(\partial T)$. In particular, for some universal constant $C>0$, 
\begin{equation*}
\|\phi^0_k\|_{C(\partial T)} + \|\phi_{k,\eta}\|_{C(\partial T)} \le C. 
\end{equation*}
Set $w^0_k = \cS[\phi^0_k]$ where $\cS$ is the single-layer potential in the free space. Then, for $z\in \frac1\eta\bT^d\setminus T$, ${\tilde\chi}^{\,\eta}_k$ is given by
\begin{equation*}
\begin{aligned}
{\tilde\chi}^{\,\eta}_k(z) &= (\cS + \eta^{d-2}\cR^\eta_1)[\phi_{k,\eta}] \\
&= w^0_k + \eta^{d-2} \cR^\eta_1[\phi^0_k] - \eta^{d-1} \cS \cR^\eta_3[\phi^0_k] + O(\eta^{2d-3})
\end{aligned}
\end{equation*}
in $L^\infty(\textstyle\frac1\eta \bT^d\setminus T)$. When restricted to $z\in \partial T$, since $N^k$ averages to zero over $\partial T$, so do $\phi^0_k$ and each of the terms in \eqref{eq:phikexpan}.
Note that $w^0_k = \cS[\phi^0_k]$ is harmonic in $\R^d\setminus \partial T$ and satisfies $\frac{\partial w^0_k}{\partial N} = N^k$ at $\partial T$. Moreover, from the definition of single-layer potential we check that $|w^0_k| \to 0$ at infinity (to check this for $d=2$, the fact $\int_{\partial T} \phi^0_k = 0$ is crucial). In other words, $w^0_k$ is the unique solution to the \emph{local} problem \eqref{eq:extnp}. Finally, thanks to the observations in Remark \ref{eq:NSL20inv} and by Proposition \ref{prop:Rexpan} we further have
\begin{equation}
\label{eq:rchietakexp}
{\tilde\chi}^{\,\eta}_k = w^0_k + \eta^d \cQ_1[\phi^0_k] - \eta^{d}\cS(\frac12 I + \cK^*)^{-1}\cQ_2[\phi^0_k] + O(\eta^{d+2} + \eta^{2(d-1)}) \qquad \text{in } C(\partial T).
\end{equation}
From this we obtain\eqref{eq:chietakexpan} (even with the next order terms). Clearly, for some $C>0$ we have $\|{\tilde\chi}^{\,\eta}_k\|_{L^\infty(\partial T)} \le C$. 

\smallskip

To prove \eqref{eq:chietakH1bdd}, via integration by parts we have
\begin{equation*}
\|\nabla {\tilde\chi}^{\,\eta}_k\|_{L^2(\frac1\eta\bT^d\setminus \ol T)}^2 = -\int_{\partial T} \frac{\partial {\tilde\chi}^{\,\eta}_k}{\partial N} {\tilde\chi}^{\,\eta}_k = \int_{\partial T} N^k {\tilde\chi}^{\,\eta}_k. 
\end{equation*}
Note the repeated $k$ above is not summed. Recall that ${\tilde\chi}^{\,\eta}_k = \cS^\eta[\phi_{k,\eta}]$ is defined everywhere in $\frac1\eta \bT^d$ and is harmonic also in $T$. Using the jump relation \eqref{eq:SDtrace}, we also have
\begin{equation*}
\|\nabla {\tilde\chi}^{\,\eta}_k\|_{L^2(T)}^2 = \int_{\partial T} \frac{\partial {\tilde\chi}^{\,\eta}_k}{\partial N}\big\vert_- {\tilde\chi}^{\,\eta}_k = \int_{\partial T} (-N^k-\phi_{k,\eta}) {\tilde\chi}^{\,\eta}_k.
\end{equation*}
Here, $\frac{\partial}{\partial N}\rvert_-$ means that the derivative is taken as a limit from the inside of $T$. The above quantities are bounded uniformly in $\eta$, and so is $\|\nabla {\tilde\chi}^{\,\eta}_k\|_{L^2(\frac1\eta\bT^d)}$. The estimate for $\|{\tilde\chi}^{\,\eta}_k\|_{L^2(\frac1\eta\bT^d)}$ then follows from the Poincar\'e-Whirtinger inequality since ${\tilde\chi}^{\,\eta}_k$, defined by the periodic single-layer potential, is mean zero in the torus. 

\smallskip

To prove the estimates \eqref{eq:chietakbdd} we verify \eqref{eq:chietakbdd1}. Combining the bound of $\|{\tilde\chi}^{\,\eta}_k\|_{L^\infty(\partial T)}$ with the maximum principle, we have $\|{\tilde\chi}^{\,\eta}_k\|_{L^\infty(\frac1\eta\bT^d\setminus \ol T)} \le C$. As for the gradient estimate, the explicit formula ${\tilde\chi}^{\,\eta}_k = \cS^\eta_\fp[\phi_{k,\eta}]$ and the uniform bound of $\phi_{k,\eta}$ lead to, for some universal constant $C$,
\begin{equation}
\left|\nabla {\tilde\chi}^{\,\eta}_k(z)\right| \le C, \quad z \in \textstyle\frac1\eta \bT^d\setminus B_1.
\end{equation}
Apply then the elliptic regularity theory for ${\tilde\chi}^{\,\eta}_k$ in $B_2\setminus \ol T$, we can also bound $|\nabla {\tilde\chi}^{\,\eta}_k|$ there and hence obtain the gradient estimate in \eqref{eq:chietakbdd1}. This completes the proof.
\end{proof}

\begin{remark}\label{rem:chirescale} In view of the relation $\chi_{k,\eta}(y) = \eta{\tilde\chi}^{\,\eta}_k(\frac{y}{\eta})$, and by rescaling, we also have
\begin{equation}
\label{eq:rschi1}
\|\nabla \chi_{k,\eta}\|_{L^2(\bT^d)} + \|\chi_{j,\eta}\|_{L^2(\bT^d)} \le C\eta^{\frac d2}.
\end{equation}
Moreover, $x \mapsto \chi_{k,\eta}(\frac{x}{\eps})$ defines an $\eps$-periodic function in $\R^d$, and by rescaling, in each $\eps$-cube (say, in $Q_\eps$), we have
\begin{equation*}
\|\nabla \chi_{k,\eta}(\textstyle\frac{\cdot}{\eps})\|_{L^2(Q_\eps)} + \|\chi_{j,\eta}(\textstyle\frac{
\cdot}{\eps})\|_{L^2(Q_\eps)} \le C(\eps \eta)^{\frac d2}.
\end{equation*}
Note that the above estimate of $\|\chi_{k,\eta}\|_{L^2(Q_\eps)}$ is better than what one gets using only $\|\chi_{k,\eta}\|_{L^\infty} \le C\eta$.
\end{remark}

\section{Proofs of the main results}

In the following two subsections, we prove Theorem \ref{thm:Aeta} and Theorem \ref{thm:uetarate} respectively.

\subsection{The expansion of the effective coefficients}\label{sec:Aexpan}
Recall the definitions of the cell problem \eqref{eq:cellp} and the rescaled cell problems \eqref{eq:rcellp}; their solutions are denoted, respectively, by $\chi_{k,\eta}$ and ${\tilde\chi}^{\,\eta}_k$. Given the informations of those functions in the previous section, the proof of Theorem \ref{thm:Aeta} becomes a straightforward computation.

Through integration by parts and rescaling, we rewrite the formula \eqref{eq:aijeta} as
\begin{equation*}
\begin{aligned}
\ol a_{ij}^\eta &= \delta_{ij} + \fint_{\bT^d \setminus \eta T} \partial_i \chi_{j,\eta}(y) dy = \delta_{ij} - \frac{1}{1-\eta^d|T|} \int_{\partial (\eta T)} N^i \chi_{j,\eta}(y) d y\\
&= \delta_{ij} - \frac{\eta^{d}}{1-\eta^d|T|} \int_{\partial T} N^i {\tilde\chi}^{\,\eta}_{j}(z) d z.
\end{aligned}
\end{equation*}
Plug in the expansion formula \eqref{eq:rchietakexp} of ${\tilde\chi}^{\,\eta}_k$, we obtain
\begin{equation}
\label{eq:abarijexpan}
\ol a_{ij}^\eta = \delta_{ij} - \eta^d \int_{\partial T} N^i w^0_k + O(\eta^{2d}).
\end{equation}
In particular, $|\ol a_{ij}^\eta - \delta_{ij}| \le C\eta^d$. Since $w^0_k$ is the unique solution to the exterior problem \eqref{eq:extnp}, we use Green's identities and check that
\begin{equation*}
-\int_{\partial T} N^i w^0_k = \int_{\partial T} \frac{\partial w^0_i}{\partial N} w^0_k = \int_{\R^d\setminus \ol T} -\nabla \cdot(w^0_k \nabla w^0_i ) = -\int_{\R^d\setminus \ol T} \nabla w^0_i \cdot \nabla w^0_k.
\end{equation*}
For the second equality, we used the fact that as $r \to \infty$,  $w^0_k \rvert_{\partial B_r} \to 0$ and $\frac{\partial w^0_i}{\partial N} \rvert_{\partial B_r}$ is of order $O(r^{-(d-1)})$. The desired formula is hence proved. 
\medskip

\subsection{Proof of the convergence rates in the homogenization}\label{sec:uetarate}

We prove Theorem \ref{thm:uetarate}. Let $u^\eps$, $\ol u^\eta$ and $u$ be the unique solutions to the problems \eqref{eq:hetnp}, \eqref{eq:hometa} and \eqref{eq:homnp} respectively. We consider the discrepancy function
\begin{equation}
\label{eq:zetadef}
\zeta^\eps = u^\eps - \ol u^\eta - \eps \chi_{k,\eta}(\frac{x}{\eps})\partial_k \ol u^\eta(x).
\end{equation}
Direct computation shows that, for $x\in \Omega^\eps$,
\begin{equation*}
\nabla \zeta^\eps = \nabla u^\eps - \nabla \ol u_\eta - (\nabla \chi_{k,\eta})(\frac{x}{\eps}) \partial_k \ol u_\eta(x) - \eps\chi_{k,\eta}(\frac{x}{\eps})\nabla \partial_k \ol u^\eta,
\end{equation*}
and
\begin{equation}
\label{eq:Dzeta1}
-\Delta \zeta^\eps = -\partial_i \left(\{\ol a_{ij}^\eta - \delta_{ij} - (\partial_i \chi_{j,\eta})(\frac{x}{\eps}) \} \partial_j \ol u^\eta \right) + \eps \partial_i \left(\chi_{j,\eta}(\frac{x}{\eps}) \partial^2_{ij} \ol u^\eta\right).
\end{equation}
Following the standard technique as carried in \cite{BLP,KLS12_ARMA}, we can represent the first term on the right in a better form as follows. Let $\theta(\eta) = \eta^d|T|$ be the volume fraction of the hole in $\bT^d\setminus (\eta \ol T)$; let $\mathbf{1}_{\bT^d\setminus \eta\overline T}$ denote the indicator function in $\bT^d$ of $\bT^d\setminus \eta \ol T$. For each fixed $j =1,\dots,d$, define the vector field
\begin{equation*}
F_j = (F_j^i), \;\quad \text{with} \;\quad F_j^i := (1-\theta(\eta)) \ol a_{ij}^\eta - \left\{\delta_{ij} + (\partial_i \chi_j)(y)\right\} \mathbf 1_{\mathbb T^d\setminus \eta\bar T},  \quad y \in \bT^d.
\end{equation*}
We compute the distributional derivative $\partial_k F^i_j$. Take a test function $\varphi \in C^\infty(\bT^d)$; we compute
\begin{equation*}
\begin{aligned}
\int_{\bT^d} -(\partial_k \varphi)F^i_j &= -\int_{\bT^d\setminus \eta\ol T} \partial_k\left(\varphi \{\delta_{ij} + \partial_i \chi_j\}\right) + \int_{\bT^d\setminus \eta\ol T} \varphi \partial_k \partial_i \chi_j\\
&= \int_{\partial T} N^k \varphi(\delta_{ij} + \partial_i \chi_j) + \int_{\bT^d\setminus \eta\ol T} \varphi \partial_k \partial_i \chi_j.
\end{aligned}
\end{equation*}
The distributional derivatives of $F_j$ can be computed easily, and due to the boundary conditions in \eqref{eq:cellp}, we also have
\begin{equation*}
\div\, F_j = \partial_\ell F^\ell_j = 0 \qquad \text{in } \bT^d.
\end{equation*}

\begin{lemma}
There exists a tensor field $\Psi = (\Psi_{kij})$, $k,i,j\in \{1,\dots,d\}$, in the unit torus $\bT^d$ and a universal constant $C>0$, such that for all $k,i,j$,
\begin{equation}
\label{eq:Psiprop}
\Psi_{kij} = - \Psi_{ikj}, \quad (F_j)^i = \partial_\ell (\Psi_{\ell i j}), \quad \|\Psi_{kij}\|_{C(\bT^d)} \le C, \; \text{and} \; \|\Psi_{kij}\|_{L^2(\bT^d)} \le C\eta^{\frac d2}.
\end{equation}
\end{lemma}
\begin{proof} Note that $\Psi$ should depend on $\eta$ but we hide it from the notations. 
For each $j$, it is clear that $\int_{\bT^d} F_j = 0$; the estimate \eqref{eq:chietakbdd} shows $F_j \in L^\infty(\bT^d)$ and, for some universal constant $C$, $\|F_j\|_{L^\infty} \le C$. 
Following the standard technique (e.g.\,\cite[Prop.\,3.1]{MR4249626}) we can find $\Psi_{kij}$ by solving
\begin{equation*}
\Delta \psi_{ij} = F^i_j \qquad \text{in } \bT^d \quad \text{with} \;\; \int_{\bT^d} \psi_{ij} = 0,
\end{equation*}
and then setting
\begin{equation*}
\Psi_{kij} = \partial_k \psi_{ij} - \partial_i \psi_{kj}.
\end{equation*}
The first identity in \eqref{eq:Psiprop} follows immediately, the second holds true since $\div F_j = 0$. The $L^\infty$ bound of $\Psi$ follows from that of $F_j$. Finally, we note that
\begin{equation*}
|F^i_j|(x) \le \begin{cases} C, &\quad \text{in } \eta \ol T,\\
C\eta^d + \partial_i \chi_j(y), &\quad \text{in } \bT^d\setminus \eta \ol T,
\end{cases}
\end{equation*}
which leads to
\begin{equation*}
\|F^i_j\|^2_{L^2(\bT^d)} \le C\eta^d + \int_{\bT^d\setminus \eta \ol T} |\partial_i \chi_{j,\eta}(y)|^2 = C\eta^d + \eta^d \int_{\frac1\eta\bT^d \setminus \ol T} |\partial_i {\tilde\chi}^{\,\eta}_j(z)|^2 dz \le C\eta^d.
\end{equation*}
Then we get
\begin{equation*}
\|\Psi_{kij}\|_{L^2(\bT^d)} \le C \max_{i,j} \|\nabla \psi_{ij}\|_{L^2(\bT^d)} \le C\max_{j} \|F_j\|_{L^2(\bT^d)} \le C\eta^{\frac d2}.
\end{equation*}
This completes the proof of the lemma.
\end{proof}

Now we derive the equations satisfied by the discrepancy function \eqref{eq:zetadef}. We claim that $\zeta^\eps$ is the (unique) solution to the following mixed boundary value problem.
\begin{equation}
\label{eq:zetambvp}
\left\{
\begin{aligned}
-\Delta \zeta^\eps &= \theta(\eta) f + \eps \partial_i \left(b_{kij}(\frac{x}{\eps}) \partial_k \partial_j \ol u^\eta\right) &\quad &\text{in } \Omega^\eps,\\
\frac{\partial \zeta^\eps}{\partial N} &= -\eps N^i b_{kij}(\frac{x}{\eps}) \partial_k \partial_j \ol u^\eta + \eps N^i \Psi_{kij}(\frac{x}{\eps}) \partial_j \partial_k \ol u^\eta &\quad &\text{in } \partial \Omega^\eps_{\rm int},\\
\zeta^\eps &= \eps \chi_{\ell, \eta}(\frac{x}{\eps}) \partial_\ell \ol u^\eta + (g-\ol u^\eta)\mathbf{1}_{\Omega \cap \partial \Oext} &\quad &\text{in } \partial \Omega^\eps_{\rm ext}.
\end{aligned}
\right.
\end{equation}
Note that $\Omega \cap \partial \Oext$ is precisely the union of the boundary of the holes that cut through $\partial \Omega$.

Firstly, using the fields $\Psi_{kij}$'s, we rewrite \eqref{eq:Dzeta1} and get
\begin{equation*}
\begin{aligned}
-\Delta \zeta^\eps &= \theta(\eta)f - \partial_i \left(F^i_j(\frac{x}{\eps}) \partial_j \ol u^\eta\right) + \eps \partial_i \left(\chi_\ell(\frac{x}{\eps}) \partial_i \partial_\ell \ol u^\eta\right)\\
&= \theta(\eta) f - \partial_i \left((\partial_k \Psi_{kij})(\frac{x}{\eps}) \partial_j \ol u^\eta\right) + \eps\partial_i\left(\chi_\ell(\frac{x}{\eps}) \partial_i \partial_\ell \ol u^\eta\right)\\
&= \theta(\eta) f-  \partial_i \left(\partial_k (\eps \Psi_{kij}(\frac{x}{\eps})) \partial_j \ol u^\eta\right) + \eps\partial_i\left(\chi_\ell(\frac{x}{\eps}) \partial_i \partial_\ell \ol u^\eta\right)\\
&= \theta(\eta) f - \partial_i \partial_k \left( \eps \Psi_{kij}(\frac{x}{\eps}) \partial_j \ol u^\eta\right) + \eps\partial_i\left(\chi_\ell(\frac{x}{\eps}) \partial_i \partial_\ell \ol u^\eta\right) + \eps\partial_i(\Psi_{kij}(\frac{x}{\eps}) \partial_k \partial_j \ol u^\eta).
\end{aligned}
\end{equation*}
In view of $\Psi_{kij} = -\Psi_{ikj}$, the second term on the right vanishes. We hence obtain the first line of \eqref{eq:zetambvp} with
\begin{equation}
\label{eq:bkijdef}
b_{kij} = \Psi_{kij} + \chi_{j,\eta} \delta_{ik}.
\end{equation}
Concerning boundary data, at the exterior boundary we check easily that the third line of \eqref{eq:zetambvp} holds;
at the interior boundary, direct computations using the boundary condition in \eqref{eq:cellp} lead to 
\begin{equation}
\label{eq:zetand}
\frac{\partial \zeta^\eps}{\partial N} = -\eps \chi_{j,\eta}(\frac{x}{\eps}) N^i \partial_i (\partial_j \ol u^\eta). 
\end{equation}
In view of \eqref{eq:bkijdef}, the above is equivalent to the second line of \eqref{eq:zetambvp}.

\medskip

We define some \emph{boundary layer} functions that are useful to deal with the mismatches between $u^\eps$ and $u$ (or $\ol u^\eta$) at $\partial \Oext$. Note that, by the modification of boundary holes in the geometric set-up in Section \ref{sec:geosetup},  in each boundary cell $Q_{k,\eps} = \eps(k+Q)$ with $k \in J_{\rm bdr}$, the enlarged hole $T^\eps_k$ contains $B_{\eps\eta}(\eps k)\cap \Omega$ and, meanwhile, is contained in the ball $B_{2\eps \eta}(\eps k)$. We can choose a smooth cut-off function $\psi^\eps_k$ with values in $[0,1]$ so that it is supported in $B_{4\eps\eta}(\eps k)$, and $\psi^\eps_k = 1$ in $B_{2\eps\eta}(\eps k)$ and $|\nabla \psi^\eps_k| \le C(\eps \eta)^{-1}$. Define
\begin{equation}
\label{eq:Phi1def}
\Phi_1(x) = \sum_{k \in J_{\rm bdr}} \psi^\eps_k(x) (g- u)(x),
\end{equation}
which is the superposition of a family of functions supported in mutually disjoint patches near $\partial \Omega$; the total number of patches is of order $O(\eps^{-d+1})$. We claim that, with $g, u \in C^{2}(\ol \Omega)$, $\Phi_1$ satisfies
\begin{equation}
\label{eq:Phi1H1}
\|\Phi_1\|_{H^1(\Omega)} \le C(\|g\|_{C^1} + \|u\|_{C^1})(\eps \eta)^{\frac d2}.
\end{equation}
To check this, we only need to estimate the $H^1$ norm of $\psi^\eps_k (g-u)$, for a fixed $k \in J_{\rm bdr}$, in the cube $B_{4\eps\eta}(\eps k)$ which contains the support of $\psi^\eps_k$. Note that
\begin{equation*}
\nabla (\psi^\eps_k (g- u)) = (\nabla \psi^\eps_k)(g - u) + \psi^\eps_k \nabla (g-u).
\end{equation*}
For each $x \in B_{4\eps\eta}(\eps k)$, let $x_0 = x_0(x)$ be a point on $\partial \Omega$ that is nearest to $x$. Clearly, $|x-x_0| \le 4\eps\eta$.  Since $u = g$ in $\partial \Omega$,  we have
\begin{equation*}
g(x) - u(x) = g(x) - g(x_0) - (u(x) -  u(x_0)).
\end{equation*}
Since $g,u \in C^2(\ol \Omega)$ and $|x-x_0| \le 4\eps\eta$, the quantity above is bounded by $(\|g\|_{C^1} + \|u\|_{C^1})\eps \eta$ uniformly in $x$, and hence
\begin{equation*}
\|g-u\|_{L^2(B_{4\eps \eta}(\eps k))} \le C(\|g\|_{C^1} + \|u\|_{C^1})(\eps \eta)^{\frac{d+2}{2}}.
\end{equation*}
Also by the uniform boundedness of $|\nabla(g-u)|$, we have
\begin{equation*}
\|\nabla (g-u)\|_{L^2(B_{4\eps \eta}(\eps k))} \le C(\|g\|_{C^1} + \|u\|_{C^1})(\eps \eta)^{\frac{d}{2}}.
\end{equation*}
Combining those computations above, we check that
\begin{equation*}
\|\psi^\eps_k (g- u)\|_{H^1(B_{4\eps \eta}(\eps k))} \le C(\eps \eta)^{-1} \|g - u\|_{L^2} + \|g-u\|_{L^2} + \|\nabla (g-u)\|_{L^2}.
\end{equation*}
Combining the estimates above, we check that \eqref{eq:Phi1H1} holds.

\medskip

For $t > 0$, let $U_t$ be the sub-domain $\{x \in \Omega \,:\, \mathrm{dist}(x,\partial \Omega) < t\}$. Note that $\partial \Oext$ is contained in $\ol U_{\eps}$. We can choose a smooth cut-off function $\psi^\eps$ with values in $[0,1]$ so that $\psi^\eps$ is supported in $\ol U_{4\eps}$ and $\psi^\eps = 1$ in $\ol U_{2\eps}$. Define
\begin{equation}
\label{eq:Phi2def}
\Phi_2(x) = \psi^\eps(x) (\eps \chi_{k,\eta}(\textstyle\frac{x}{\eps}) \partial_k u(x)), \quad x\in \ol \Omega.
\end{equation}
We claim that, for some universal constant $C > 0$,
\begin{equation}
\label{eq:Phi2H1}
\|\Phi_2 \|_{H^1(\Omega)} \le C\|u\|_{C^2}(\eps \eta^d)^{\frac 12}.
\end{equation}
We compute and check
\begin{equation*}
\nabla \Phi_2 = \eps (\nabla \psi^\eps) \chi_{k,\eta}(\textstyle\frac{x}{\eps}) \partial_k u + \psi^\eps (\nabla \chi_{k,\eta})(\textstyle \frac{x}{\eps}) \partial_k u + \eps \psi^\eps \chi_{k,\eta}(\textstyle\frac{x}{\eps}) \partial_k \nabla u.
\end{equation*}
Note that $U_{4\eps}$ is covered by a collection of $\eps$-cubes near $\partial \Omega$, and the total number is of order $O(\eps^{-d+1})$. Hence, in view of the estimates in Remark \ref{rem:chirescale}, we have
\begin{equation*}
\|\chi_{k,\eta}(\textstyle\frac{\cdot}{\eps})\|_{H^1(U_{4\eps})} \le \left[C \eps^{-d+1} (\eps \eta)^d \right]^{\frac12} \le C(\eps \eta^d)^{\frac 12}.
\end{equation*}
Then by the uniform bound of $\|u\|_{C^2(\Omega)}$ and the formula of $\nabla \Phi_2$ above, we get
\begin{equation*}
\|\Phi_2\|_{H^1} \le \sum_k C(\eps+1)\|u\|_{C^2}\|\chi_{k,\eta}\|_{H^1(U_{4\eps})}.
\end{equation*}
Combining those results above, we obtain \eqref{eq:Phi2H1}.

Clearly, if we replace $u$ by $\ol u^\eta$ in the definition of those boundary layer functions $\Phi_1$ and $\Phi_2$, the above results still hold. This is because, due to the uniform ellipticity of $\ol A(\eta)$ (for sufficiently small $\eta$), $\|\ol u^\eta\|_{C^2}$ can be bounded uniformly in $\eta$.

We observe that the problem \eqref{eq:zetambvp} is of the standard form \eqref{eq:mbvppd} with
\begin{equation*}
\begin{aligned}
f_0 &= \theta(\eta) f, \; &F &= (\eps b_{kij}(\frac{x}{\eps})\partial_k \partial_j \ol u^\eta)_i, \\
h &= \eps N^i \Psi_{kij}(\frac{x}{\eps})\partial_j \partial_k \ol u^\eta, \; &g_0 &= \Phi_1 + \Phi_2.
\end{aligned}
\end{equation*}
Invoking the energy estimate in Theorem \ref{thm:energyest}, the easy $L^2$ bound on $f_0$ and the estimates \eqref{eq:Phi1H1}-\eqref{eq:Phi2H1}, we have
\begin{equation*}
\|\nabla \zeta^\eps\|_{L^2(\Omega^\eps)} \le C\left\{ \eta^d + \sqrt{\eps} \eta^{\frac d2} + \|F\|_{L^2(\Omega^\eps)} + \kappa_{\eps,\eta}\|h\|_{L^2(\partial \Oint)}\right\}. 
\end{equation*}
In view of the uniform in $\eta$ estimate of $\|\ol u^\eta\|_{C^2(\Omega)}$, the periodicity of $b = (b_{kij})$ and its representation \eqref{eq:bkijdef}, the estimates \eqref{eq:Psiprop} and \eqref{eq:rschi1}, and by the usual rescaling, we have
\begin{equation*}
\|F\|_{L^2(\Omega^\eps)} \le C\eps \|b\|_{L^2(\bT^d)} \le C\eps\left\{\|\Psi\|_{L^2(\bT^d)} + \|\chi\|_{L^2(\bT^d)}\right\} \le C\eps \eta^{\frac d2}.
\end{equation*}
In a similar way, we also have
\begin{equation*}
\|h\|_{L^2(\partial \Oint)} \le C\sqrt{\eps} \|\Psi\|_{L^\infty} \|\ol u^\eta\|_{C^2} \|1\|_{L^2(\partial (\eta T))} \le C\left(\eps \eta^{d-1}\right)^{\frac12}.
\end{equation*}
Note that $\zeta^\eps - (\Phi_1 + \Phi_2)$ belongs to the functional space $V_\eps$ defined in \eqref{eq:Vepsdef}. Hence, by the Poincar\'e inequality \eqref{eq:pipd}, we get
\begin{equation*}
\|\zeta^\eps - (\Phi_1 + \Phi_2)\|_{H^1(\Omega^\eps)} \le \|\nabla \zeta^\eps\|_{L^2(\Omega^\eps)} + \|\nabla(\Phi_1+\Phi_2)\|_{L^2(\Omega)}.
\end{equation*}
We use \eqref{eq:Phi1H1} and \eqref{eq:Phi2H1} to control $\|\nabla(\Phi_1+\Phi_2)\|_{L^2}$ and invoke elliptic regularity theory to get 
$$
\|\ol u^\eta\|_{C^{2,\alpha}} \le C(\|f\|_{C^\alpha} + \|g\|_{C^{2,\alpha}}).
$$
Those estimates lead to 
\begin{equation*}
\begin{aligned}
\|\zeta^\eps\|_{H^1(\Omega^\eps)} &\le C\left\{\|\Phi_1+\Phi_2\|_{H^1(\Omega)} + \|\nabla \zeta^\eps\|_{L^2(\Omega^\eps)}\right\}\\
&\le C(\|f\|_{C^\alpha} + \|g\|_{C^{2,\alpha}})\{\eta^d + \sqrt{\eps}\eta^{\frac d2} + \eps \eta^{\frac d2} + \kappa_{\eps,\eta} (\eps \eta^{d-1})^{\frac12}\}.
\end{aligned}
\end{equation*}
Fix $d\ge 3$ first. From the definition \eqref{eq:kappadef}, we know that the last term on the right hand side is of order $O(\eps\eta^{\frac d2})$ if and only if $1/\sigma_\eps^2 \lesssim 1$, then the quantity above is of order $O(\sqrt{\eps}\eta^{\frac d2})$. When $1/\sigma_\eps^2 \gg 1$, then the last term is of order $O(\eta^{d-1})$ and we need to compare $\eta^{d-1}$ with $\sqrt{\eps}\eta^{\frac d2}$. We see that $O(\eta^{d-1})$ dominates if and only if $\eta^{d-2}/\eps \gg 1$, or equivalently, when $\eta^{d-2}/\sigma_\eps^2 \gg 1$ since $\eta^{d-1}/\sigma^2_\eps = (\eta^{d-2}/\eps)^2$. The case of $d=2$ can be checked  similarly. This concludes the proof of Theorem \ref{thm:uetarate}. 

\begin{remark}\label{rem:altmethod} We provide a simpler argument that proves \eqref{eq:uepsuH1}. For any $w\in V_\eps$, let $\tilde w = \mathcal{E} w \in H^1_0(\Omega)$ denote the extension of $w$ in $\Omega$ in Proposition \ref{prop:ext}. From equations of $u^\eps$ and $u$, we have
\begin{equation*}
\int_{\Omega^\eps} \nabla u^\eps \cdot \nabla w = \int_{\Omega^\eps} f\tilde w, \quad \text{and} \;\; \int_{\Omega} \nabla u \cdot \nabla \tilde w = \int_{\Omega} f\tilde w, \quad \forall w \in V_\eps.
\end{equation*}
Consider the discrepancy function $v^\eps = u^\eps - u$. By the identities above,  for all $w \in V_\eps$,
\begin{equation*}
\int_{\Omega^\eps} \nabla v^\eps \cdot \nabla w = \int_{\Omega^\eps} f\tilde w - \int_{\Omega^\eps} \nabla u \cdot \nabla \tilde w = \int_{\Omega\setminus \Omega^\eps}  \nabla u \cdot \nabla \tilde w - f\tilde w.
\end{equation*}
Note that $v^\eps \in H^1(\Omega^\eps)$ and its trace in $\partial \Oext$ agrees with $\Phi_1$ defined in \eqref{eq:Phi1def}, so setting $w = v^\eps - \Phi_1$ then $w \in V_\eps$. We then get
\begin{equation*}
\|\nabla w\|^2_{L^2(\Omega^\eps)} = -\int_{\Omega^\eps} \nabla \Phi_1 \cdot \nabla w + \int_{\Omega\setminus \Omega^\eps}  \nabla u \cdot \nabla \tilde w - f\tilde w.
\end{equation*}
Using the Poincar\'e inequality \eqref{eq:pipd} for $\tilde w$ and properties of the extension operators, and noting that $\|1\|_{L^2(\Omega\setminus \Omega^\eps)}$ is of order $O(\eta^{\frac d2})$, we get
\begin{equation*}
\|w\|_{H^1(\Omega^\eps)} \le C\|\nabla w\|_{L^2(\Omega^\eps)} \le C\{\|\nabla \Phi_1\|_{L^2} + \eta^{\frac d2}(\|\nabla u\|_{L^\infty} + \|f\|_{L^\infty})\}.
\end{equation*}
Since $v^\eps = w + \Phi_1$, we can replace $w$ on the left by $v^\eps$. Thanks to the estimate \eqref{eq:Phi1H1}, we obtain
\begin{equation*}
\|u^\eps - u\|_{H^1(\Omega^\eps)} \le C(1+\eps^{\frac d2})\eta^{\frac d2}.
\end{equation*}
This gives a short proof for the estimate \eqref{eq:uepsuH1}.
\end{remark}

\begin{remark}\label{rem:uetau} We briefly outline the proof of Corollary \ref{coro:uetau}. From \eqref{eq:hometa} we have
\begin{equation*}
\int_\Omega \ol A(\eta) \nabla (\ol u^\eta - u) \cdot \nabla (\ol u^\eta - u) = \int_\Omega f(\ol u^\eta - u) - \int_\Omega \nabla u \cdot (\ol u^\eta - u) + \int_\Omega (I - \ol A(\eta))\nabla u \cdot \nabla (\ol u^\eta - u).
\end{equation*}
The first two terms on the right cancel. In view of $|\ol A(\eta) - I|\le C\eta^d$ and that $\|\nabla u\|_{L^2}$ is bounded, we get \eqref{eq:uetauH1}. With more regular data we can control $u - \ol u^\eta$ in $C^{2,\alpha}(\ol \Omega)$. Indeed, the difference function $u - \ol u^\eta$ is of class $C^{2,\alpha}(\ol \Omega)$ and satisfies
\begin{equation*}
\left\{\begin{aligned}
&-\nabla \cdot (\ol A(\eta) \nabla (\ol u^\eta - u)) = h, &\quad \text{in } \Omega,\\
&\ol u^\eta - u = 0, &\quad \text{in } \partial \Omega,
\end{aligned}
\right.
\end{equation*}
where $h = \nabla \cdot ((\ol A(\eta) - I)\nabla u)$. In view of \eqref{eq:aijeta} we have $\|h\|_{C^\alpha(\Omega)} \le C\eta^d\|u\|_{C^{2,\alpha}}$. Then \eqref{eq:uetauC2} follows from the standard elliptic estimate.
\end{remark}

\section*{acknowledgements}
The author acknowledges the support from the NSF of China under Grants No.\,11871300. The author is also grateful for the suggestions and corrections made by the anonymous referee.

\appendix

\section{Useful tools for functions in perforated domains}

\subsection{Extension operators}

Let $\Omega^\eps$ as defined in Section \ref{sec:geosetup}. We use $\tilde \Omega^\eps$ to denote the domain obtained by filling the \emph{interior} holes. Note that the (modified) boundary holes are left untouched and we see $\partial \tilde \Omega^\eps = \partial \Oext$.

\begin{proposition}\label{prop:ext} There exists an extension operator $\mathcal{E}: H^1(\Omega^\eps) \to H^1(\tilde\Omega^\eps)$, and a universal constant $C > 0$ such that, for any $w \in H^1(\Omega^\eps)$, $\mathcal{E} w$ is in $H^1(\tilde\Omega^\eps)$ and satisfies
\begin{equation}
\label{eq:extbdd}
\begin{aligned}
&\mathcal{E} w = w \, \text{ in } \Omega^\eps, \\
&\|\mathcal{E} w\|_{L^2(\tilde\Omega^\eps)} \le C \|w\|_{L^2(\Omega^\eps)}, \qquad \|\nabla(\mathcal{E}w)\|_{L^2(\tilde\Omega^\eps)} \le C\|\nabla w\|_{L^2(\Omega^\eps)}.
\end{aligned}
\end{equation}
If $w$ is also in $V_\eps$, i.e.\,$w$ vanishing in $\partial \Oext$, then there exists an extension such that $\mathcal{E} w \in H^1_0(\Omega)$ and the above estimates hold with $\tilde\Omega^\eps$ replaced by $\Omega$. 
\end{proposition}

\begin{proof} Such extensions are classic; see e.g.\,\cite{MR548785,AGGJS,J-CMS,MR974289}. We recall the key ingredients. Without loss of generality we assume $T$ is simply connected subset of $B = B_1(0)$. Basic Sobolev spaces theory allows us to find an extension operator $P$ that maps $H^1(B\setminus \ol T)$ to $H^1(B)$ and a universal $C>0$ so that, for any $v \in H^1(B\setminus \ol T)$, $Pv \in H^1(B)$ satisfies
\begin{equation*}
\begin{aligned}
&Pv = v \; \text{ in } B \setminus \ol T,\\
&\|Pv\|_{L^2(B)} \le \|v\|_{L^2(B\setminus \ol T)}, \quad \|\nabla(Pv)\|_{L^2(B)} \le \|v\|_{H^1(B\setminus \ol T)}.
\end{aligned}
\end{equation*}
Then for $w \in H^1(B\setminus \ol T)$ above, we apply $P$ to the mean-zero function $w - \fint_{B\setminus \ol T} w$ and then add the mean term back; that is,
\begin{equation*}
\mathcal{E} w(y) = \textstyle\fint_{B\setminus \ol T} w + P\left(w- \fint_{B\setminus \ol T} w\right)(y), \qquad y\in B.
\end{equation*}
We check that $\mathcal{E} w \in H^1(B)$ and $\mathcal{E}w = w$ in $B\setminus \ol T$. In view of the inequalities
\begin{equation*}
\left\lvert\textstyle\fint_{B\setminus \ol T} w\right\rvert \le C\|w\|_{L^2(B_1\setminus \ol T)} \quad \text{and} \quad \|w-\textstyle\fint_{B\setminus \ol T} w\|_{L^2(B\setminus \ol T)} \le C\|\nabla w\|_{L^2(B\setminus \ol T)},
\end{equation*} 
we also have
\begin{equation*}
\|\mathcal{E} w\|_{L^2(B)} \le C\|w\|_{L^2(B\setminus \ol T)}, \quad \|\nabla(\mathcal{E} w)\|_{L^2(B)} \le C\|\nabla w\|_{L^2(B\setminus \ol T)}.
\end{equation*}
For each of the inequalities, the norms on both sides involve the same order of derivatives and integrations, and are hence scaling invariant. Therefore, given $w\in \Omega^\eps$, we can extend $w$ in each $\eps$-cell $\eps (k+ Q_1)$ by rescaling. More precisely, focusing inside $\eps (k + \eta B)$, we define
\begin{equation*}
\mathcal{E} w(y) = \left(\mathcal{E}\left[w(\eps k + \eps \eta \cdot)\right]\right)\left(\frac{\cdot - \eps k}{\eps\eta}\right), \qquad y \in \eps (k + \eta B).
\end{equation*}
Glue those extensions and keep the value of $w$ elsewhere; we get $\mathcal{E}$. The estimates \eqref{eq:extbdd} hold due to the aforementioned scaling invariance. In case $w \in V_\eps$, we set $w$ by zero in the boundary holes. The desired estimates hold clearly.
\end{proof}

As an application, those extension operators can be used to prove the following Poincar\'e type inequality.

\begin{corollary} There is a universal constant $C>0$ such that
\begin{equation}
\label{eq:pipd}
\|w\|_{L^2(\Omega^\eps)} \le C\|\nabla w\|_{L^2(\Omega^\eps)}, \qquad \forall w \in V_\eps.
\end{equation}
\end{corollary}
\begin{proof} We apply the extension operator in Proposition \ref{prop:ext} to get $\mathcal{E} w\in H^1_0(\Omega)$, to which the usual Poincar\'e inequality can be applied. \eqref{eq:pipd} then follows from the last inequality of \eqref{eq:extbdd}.
\end{proof}

\begin{remark} We also refer to \cite{AllMur} for an alternative proof of \eqref{eq:pipd} which does not rely on the extension operators and hence applies to more general, namely macroscopically connected, holes. 
\end{remark}

\subsection{Trace type inequalities}

In the proof of the basic energy estimate in Section \ref{sec:energy}, we need the following trace type estimate, where the parameter $\kappa_{\eps,\eta}$ is defined in \eqref{eq:kappadef}. 

\begin{proposition}\label{prop:trace} Assume that the geometric set-up conditions {\upshape(A)} hold. There is a universal constant $C > 0$ so that, for all $\eps$ and $\eta$, the following holds:
\begin{equation}
\label{eq:trimbvp}
\|w\|_{L^2(\partial \Omega^\eps_{\rm int})} \le C \kappa_{\eps,\eta} \|w\|_{H^1(\Omega^\eps)}, \qquad \forall w \in H^1(\Omega^\eps).
\end{equation}
\end{proposition}

We refer to Lemma 2.1 of Conca and Donato \cite{MR974289}. The proof there is based on a careful computation of the trace and volume integrals inside each $\eps$-period. The above estimates are much sharper than what one gets from a naive rescaling of the usual trace estimate, as the latter only yields a rough bound of order $O((\eps \eta)^{-\frac12})$.


\begin{thebibliography}{10}

\bibitem{Allaire91-1}
G.~Allaire.
\newblock Homogenization of the {N}avier-{S}tokes equations in open sets
  perforated with tiny holes. {I}. {A}bstract framework, a volume distribution
  of holes.
\newblock {\em Arch. Rational Mech. Anal.}, 113(3):209--259, 1990.

\bibitem{Allaire91-2}
G.~Allaire.
\newblock Homogenization of the {N}avier-{S}tokes equations in open sets
  perforated with tiny holes. {II}. {N}oncritical sizes of the holes for a
  volume distribution and a surface distribution of holes.
\newblock {\em Arch. Rational Mech. Anal.}, 113(3):261--298, 1990.

\bibitem{AllMur}
G.~Allaire and F.~Murat.
\newblock Homogenization of the {N}eumann problem with nonisolated holes.
\newblock {\em Asymptotic Anal.}, 7(2):81--95, 1993.
\newblock With an appendix written jointly with A. K. Nandakumar.

\bibitem{MR3025042}
H.~Ammari, P.~Garapon, H.~Kang, and H.~Lee.
\newblock Effective viscosity properties of dilute suspensions of arbitrarily
  shaped particles.
\newblock {\em Asymptot. Anal.}, 80(3-4):189--211, 2012.

\bibitem{AGGJS}
H.~Ammari, J.~Garnier, L.~Giovangigli, W.~Jing, and J.-K. Seo.
\newblock Spectroscopic imaging of a dilute cell suspension.
\newblock {\em J. Math. Pures Appl. (9)}, 105(5):603--661, 2016.

\bibitem{AmmKan}
H.~Ammari and H.~Kang.
\newblock {\em Polarization and moment tensors}, volume 162 of {\em Applied
  Mathematical Sciences}.
\newblock Springer, New York, 2007.
\newblock With applications to inverse problems and effective medium theory.

\bibitem{MR2244590}
H.~Ammari, H.~Kang, and M.~Lim.
\newblock Effective parameters of elastic composites.
\newblock {\em Indiana Univ. Math. J.}, 55(3):903--922, 2006.

\bibitem{MR2129229}
H.~Ammari, H.~Kang, and K.~Touibi.
\newblock Boundary layer techniques for deriving the effective properties of
  composite materials.
\newblock {\em Asymptot. Anal.}, 41(2):119--140, 2005.

\bibitem{MR773850}
H.~Attouch.
\newblock {\em Variational convergence for functions and operators}.
\newblock Applicable Mathematics Series. Pitman (Advanced Publishing Program),
  Boston, MA, 1984.

\bibitem{AL87_Lp}
M.~Avellaneda and F.-H. Lin.
\newblock Homogenization of elliptic problems with {$L^p$} boundary data.
\newblock {\em Appl. Math. Optim.}, 15(2):93--107, 1987.

\bibitem{AL91_Lp}
M.~Avellaneda and F.-H. Lin.
\newblock {$L^p$} bounds on singular integrals in homogenization.
\newblock {\em Comm. Pure Appl. Math.}, 44(8-9):897--910, 1991.

\bibitem{BLP}
A.~Bensoussan, J.-L. Lions, and G.~C. Papanicolaou.
\newblock Boundary layers and homogenization of transport processes.
\newblock {\em Publ. Res. Inst. Math. Sci.}, 15(1):53--157, 1979.

\bibitem{CioMur-1}
D.~Cioranescu and F.~Murat.
\newblock Un terme \'{e}trange venu d'ailleurs.
\newblock In {\em Nonlinear partial differential equations and their
  applications. {C}oll\`ege de {F}rance {S}eminar, {V}ol. {II} ({P}aris,
  1979/1980)}, volume~60 of {\em Res. Notes in Math.}, pages 98--138, 389--390.
  Pitman, Boston, Mass.-London, 1982.

\bibitem{MR548785}
D.~Cioranescu and J.~S.~J. Paulin.
\newblock Homogenization in open sets with holes.
\newblock {\em J. Math. Anal. Appl.}, 71(2):590--607, 1979.

\bibitem{MR974289}
C.~Conca and P.~Donato.
\newblock Nonhomogeneous {N}eumann problems in domains with small holes.
\newblock {\em RAIRO Mod\'{e}l. Math. Anal. Num\'{e}r.}, 22(4):561--607, 1988.

\bibitem{feppon:hal-02518528}
F.~Feppon.
\newblock {High order homogenization of the Poisson equation in a perforated
  periodic domain}.
\newblock Preprint, Mar. 2020.

\bibitem{MR4259909}
F.~Feppon.
\newblock High order homogenization of the {S}tokes system in a periodic porous
  medium.
\newblock {\em SIAM J. Math. Anal.}, 53(3):2890--2924, 2021.

\bibitem{feppon:hal-03098222}
F.~Feppon and W.~Jing.
\newblock {High order homogenized Stokes models captureall three regimes}.
\newblock Preprint, Jan. 2021.

\bibitem{Folland}
G.~B. Folland.
\newblock {\em Introduction to partial differential equations}.
\newblock Princeton University Press, Princeton, NJ, second edition, 1995.

\bibitem{MR4280836}
D.~G\'{e}rard-Varet.
\newblock Derivation of the {B}atchelor-{G}reen formula for random suspensions.
\newblock {\em J. Math. Pures Appl. (9)}, 152:211--250, 2021.

\bibitem{DGV-dilute}
D.~G\'{e}rard-Varet.
\newblock A simple justification of effective models for conducting or fluid
  media with dilute spherical inclusions.
\newblock {\em Asymptotic Analysis}, to appear.

\bibitem{MR4290385}
A.~Giunti.
\newblock Derivation of {D}arcy's law in randomly perforated domains.
\newblock {\em Calc. Var. Partial Differential Equations}, 60(5):Paper No. 172,
  2021.

\bibitem{MR4020526}
A.~Giunti and R.~M. H\"{o}fer.
\newblock Homogenisation for the {S}tokes equations in randomly perforated
  domains under almost minimal assumptions on the size of the holes.
\newblock {\em Ann. Inst. H. Poincar\'{e} Anal. Non Lin\'{e}aire},
  36(7):1829--1868, 2019.

\bibitem{J-CMS}
W.~Jing.
\newblock Homogenization of randomly deformed conductivity resistant membranes.
\newblock {\em Commun. Math. Sci.}, 14(5):1237--1268, 2016.

\bibitem{MR4075336}
W.~Jing.
\newblock A unified homogenization approach for the {D}irichlet problem in
  perforated domains.
\newblock {\em SIAM J. Math. Anal.}, 52(2):1192--1220, 2020.

\bibitem{MR4172687}
W.~Jing.
\newblock Layer potentials for {L}am\'{e} systems and homogenization of
  perforated elastic medium with clamped holes.
\newblock {\em Calc. Var. Partial Differential Equations}, 60(1):Paper No. 2,
  32, 2021.

\bibitem{JLP-stokes}
W.~Jing, Y.~Lu, and C.~Prange.
\newblock Stokes potentials and applications in homogenization problems in
  perforated domains, in preparation.

\bibitem{Kacimi_Murat}
H.~Kacimi and F.~Murat.
\newblock Estimation de l'erreur dans des probl\`emes de {D}irichlet o\`u
  apparait un terme \'{e}trange.
\newblock In {\em Partial differential equations and the calculus of
  variations, {V}ol. {II}}, volume~2 of {\em Progr. Nonlinear Differential
  Equations Appl.}, pages 661--696. Birkh\"{a}user Boston, Boston, MA, 1989.

\bibitem{KLS12_ARMA}
C.~E. Kenig, F.~Lin, and Z.~Shen.
\newblock Convergence rates in {$L^2$} for elliptic homogenization problems.
\newblock {\em Arch. Ration. Mech. Anal.}, 203(3):1009--1036, 2012.

\bibitem{KLS13_Neumann}
C.~E. Kenig, F.~Lin, and Z.~Shen.
\newblock Homogenization of elliptic systems with {N}eumann boundary
  conditions.
\newblock {\em J. Amer. Math. Soc.}, 26(4):901--937, 2013.

\bibitem{KLS14_GN}
C.~E. Kenig, F.~Lin, and Z.~Shen.
\newblock Periodic homogenization of {G}reen and {N}eumann functions.
\newblock {\em Comm. Pure Appl. Math.}, 67(8):1219--1262, 2014.

\bibitem{MR1195131}
O.~A. Ole\u{\i}nik, A.~S. Shamaev, and G.~A. Yosifian.
\newblock {\em Mathematical problems in elasticity and homogenization},
  volume~26 of {\em Studies in Mathematics and its Applications}.
\newblock North-Holland Publishing Co., Amsterdam, 1992.

\bibitem{MR609184}
G.~C. Papanicolaou and S.~R.~S. Varadhan.
\newblock Diffusion in regions with many small holes.
\newblock In {\em Stochastic differential systems ({P}roc. {IFIP}-{WG} 7/1
  {W}orking {C}onf., {V}ilnius, 1978)}, volume~25 of {\em Lecture Notes in
  Control and Information Sci.}, pages 190--206. Springer, Berlin-New York,
  1980.

\bibitem{MR377303}
J.~Rauch and M.~Taylor.
\newblock Potential and scattering theory on wildly perturbed domains.
\newblock {\em J. Functional Analysis}, 18:27--59, 1975.

\bibitem{MR3659366}
B.~C. Russell.
\newblock Homogenization in perforated domains and interior {L}ipschitz
  estimates.
\newblock {\em J. Differential Equations}, 263(6):3396--3418, 2017.

\bibitem{shen2020sharp}
Z.~Shen.
\newblock Sharp convergence rates for darcy's law, 2020, to appear in {\em Comm. Partial Differential Equations}.

\bibitem{MR4249626}
Z.~Shen.
\newblock Quantitative homogenization of elliptic operators with periodic
  coefficients.
\newblock In {\em Harmonic analysis and applications}, volume~27 of {\em
  IAS/Park City Math. Ser.}, pages 73--130. Amer. Math. Soc., [Providence], RI,
  2020.

\bibitem{shen2021compactness}
Z.~Shen.
\newblock Compactness and large-scale regularity for darcy's law, 2021, arXiv:2104.05074.

\bibitem{shen2021homogenization}
Z.~Shen.
\newblock Homogenization of boundary value problems in perforated lipschitz
  domains, 2021, arXiv:2105.13234.

\bibitem{MR4240768}
L.~Wang, Q.~Xu, and P.~Zhao.
\newblock Convergence rates for linear elasticity systems on perforated
  domains.
\newblock {\em Calc. Var. Partial Differential Equations}, 60(2):Paper No. 74,
  51, 2021.

\end{thebibliography}

\end{document}